\documentclass[12pt,oneside]{amsart}
\usepackage[utf8]{inputenc}

\usepackage[english]{babel}
\usepackage[a4paper,textwidth=125mm,textheight=195mm,centering]{geometry}
\usepackage{amsmath,amsthm,amssymb,mathrsfs,MnSymbol}
\usepackage{mathtools}
\usepackage{graphicx}
\usepackage[colorlinks=true, allcolors=blue]{hyperref}
\usepackage{tikz-cd,adjustbox}
\usepackage{spverbatim,comment}
\usepackage[shortlabels]{enumitem}
\usepackage{caption}
\usepackage{subcaption}
\usepackage{stmaryrd}  %\sslash for closer // in GIT quotient

\newcounter{intro}

\newtheorem{intro-conjecture}[intro]{Conjecture}
\newtheorem{intro-corollary}[intro]{Corollary}
\newtheorem{intro-theorem}[intro]{Theorem}
\newtheorem{intro-proposition}[intro]{Proposition}

\theoremstyle{plain}
\newtheorem{theorem}{Theorem}[section]
\newtheorem{lemma}[theorem]{Lemma}
\newtheorem{lemma-def}[theorem]{Lemma-Definition}
\newtheorem{proposition}[theorem]{Proposition}
\newtheorem{corollary}[theorem]{Corollary}

\theoremstyle{definition}
\newtheorem{definition}[theorem]{Definition}

\theoremstyle{remark}
\newtheorem{remark}[theorem]{Remark}
\newtheorem*{remark*}{Remark}

\newtheorem{example}[theorem]{Example}
\newtheorem*{example*}{Example}
\newtheorem{question}[theorem]{Question}
\newtheorem{problem}[theorem]{Problem}

\usepackage[alphabetic,initials]{amsrefs}

\newcommand{\Rom}[1]{\uppercase\expandafter{\romannumeral#1}}

\makeatletter
\let\oldu\u
\def\u{\mathfrak u}
\makeatother

\newcommand{\E}{\mathcal E}
\newcommand{\F}{\mathcal F}
\newcommand{\G}{\mathcal G}
\renewcommand{\H}{\mathcal H}
\newcommand{\I}{\mathcal I}
\newcommand{\K}{\mathcal K}

\newcommand{\N}{\mathcal N}
\renewcommand{\O}{\mathcal O}

\newcommand{\R}{\mathcal R}

\newcommand{\X}{\mathcal X}

\newcommand{\NK}{\mathcal N \mathcal K} %local NK group

\newcommand{\bR}{{\bf R}}

\renewcommand{\AA}{\mathbb A}
\newcommand{\CC}{\mathbb C}

\newcommand{\HH}{\mathbb H}

\newcommand{\LL}{\mathbb L}
\newcommand{\NN}{\mathbb N}
\newcommand{\PP}{\mathbb P}
\newcommand{\QQ}{\mathbb Q}

\DeclareMathOperator{\coker}{coker}

\newcommand{\xto}{\xrightarrow} %use \xto^{\sim} to add texts/symbols on arrows
\newcommand{\blank}{{-}} %placeholder when defining functors

\newcommand*{\sHom}{\mathscr{H}\kern -.5pt om}

\newcommand{\DB}{\underline{\Omega}} %for du Bois complex
\DeclareMathOperator{\Gr}{Gr}
\DeclareMathOperator{\DR}{DR}

\DeclareMathOperator{\Cone}{Cone}
\DeclareMathOperator{\tors}{tors}

\newcommand{\dR}{\widehat{\rm dR}}

\DeclareMathOperator{\sing}{sing}
\DeclareMathOperator{\supp}{supp}
\DeclareMathOperator{\Spec}{Spec}

\DeclareMathOperator{\Pic}{Pic}

\DeclareMathOperator{\codim}{codim}

\DeclareMathOperator{\depth}{depth}

\DeclareMathOperator{\gr}{gr}

\DeclareMathOperator{\an}{an}

\DeclareMathOperator{\Fil}{Fil}

\newcommand{\HC}{{\bf HC}}

\newcommand{\Mustata}{Musta\c{t}\oldu{a}}

\DeclarePairedDelimiter\ceil{\lceil}{\rceil}
\DeclarePairedDelimiter\floor{\lfloor}{\rfloor}
\DeclarePairedDelimiter\abs{\lvert}{\rvert}

\makeatletter
\let\oldabs\abs
\def\abs{\@ifstar{\oldabs}{\oldabs*}}
\makeatother

\theoremstyle{definition} % just in case the style had changed
\newcommand{\thistheoremname}{}
\newtheorem*{genericthm}{\thistheoremname}

\title{On Higher Du Bois Singularities and $K$-Regularity}
\author{Wanchun Shen}
\address{Department of Mathematics, Harvard University, 1 Oxford Street, Cambridge, MA 02138, USA}
\email{wshen@math.harvard.edu}
\date{}

\begin{document}

\begin{abstract}
    We study the relationship between higher Du Bois singularities and $K$-regularity, a notion that measures the $\AA^1$-invariance of the algebraic $K$-groups. Building on this relationship, we establish a strengthened form of Vorst’s conjecture for local complete intersections in characteristic zero. Our work also provides tools to construct new examples that illustrate various phenomena in the study of $K$-regularity. The main inputs for our results are vanishing theorems for the Du Bois complexes.
\end{abstract}

\subjclass[2020]{14B05, 19E08}

\keywords{higher Du Bois singularities, $K$-regularity, Vorst's conjecture, Du Bois complex, minimal exponent}

\maketitle
\tableofcontents

Throughout, we assume $F$ is a field of characteristic zero. By a variety, we mean a reduced, separated scheme of finite type over $F$.

\section{Introduction}
Understanding the $\AA^1$-invariance of $K$-groups is a central problem in algebraic $K$-theory. A natural way to quantify the failure of $\AA^1$-invariance is by studying the cokernels of the maps
\[\varphi_{m,r}: K_m(X)\to K_m(X\times \AA^r)\]
for all integers $m$ and $r\ge 0$. A scheme $X$ is said to be \textit{$K_m$-regular} if the maps $\varphi_{m,r}$ are isomorphisms for all $r\ge 0$. It is a theorem of Quillen that regular schemes are $K_m$-regular for all $m$. In general, the extent to which $K_m(X)$ fails to be $\AA^1$-invariant provides a measure of the singularities of $X$, and this is captured in the notion of $K$-regularity. 

This paper is devoted to applications of Hodge-theoretic methods to the study of $K$-regularity. As an illustration of the kind of results obtained by these methods, we highlight a few consequences, the first one being a numerical characterization of $K_m$-regularity in terms of the \textit{minimal exponent} $\tilde{\alpha}_X$ introduced by Saito \cite{Saito-min-exponent}. See e.g. \cite{Popa-lecture-notes}*{Section 3.7} for a nice introduction to this numerical invariant that refines the log canonical threshold of $X$:

\begin{intro-theorem}\label{intro-thm:char-of-K-regularity}
    Let $X$ be a complex hypersurface of dimension $d$ with isolated singularities. Then
    \[X \text{ is $K_m$-regular} \iff \tilde{\alpha}_X \ge \ceil{\frac{m+d}{2}}.\] 
\end{intro-theorem}

The following example illustrates the use of Theorem \ref{intro-thm:char-of-K-regularity} to detect $K$-regularity:

\begin{example*}
    Consider the hypersurface $X = (x^2+y^3+z^7=0)\subset \AA^{3}_{\CC}$. By \cite{Saito-homogeneous-min-exp}*{4.1.5}, the minimal exponent of $X$ is given by $\tilde{\alpha}_X = \frac{1}{2}+\frac{1}{3}+ \frac{1}{7}=\frac{41}{42}$. Thus, Theorem \ref{intro-thm:char-of-K-regularity} says that
    \begin{align*}
        \text{$X$ is $K_m$-regular} 
        &\iff \frac{41}{42}\ge \ceil{\frac{m+2}{2}} \\
        &\iff m\le -2
    \end{align*}
    Thus $X$ is $K_{-2}$-regular, but not $K_m$-regular for $m\ge -1$. This recovers \cite{Weibel-surface}*{Example 6.1}.
\end{example*}

Moreover, we have the following regularity criterion:

\begin{intro-theorem}\label{intro-thm:regularity}
Let $X$ be an affine local complete intersection. If $X$ is $K_2$-regular, then it is regular.
\end{intro-theorem}

Interest in regularity criteria of this kind goes back to the work of Vorst \cite{Vorst}, who conjectured that for an affine, finite type $F$-scheme of dimension $d$, regularity and $K_{d+1}$-regularity are equivalent. This conjecture was confirmed in \cite{K-Vorst} (as usual, we assume the base field has characteristic zero; over perfect fields of characteristic $p$, Vorst's conjecture was confirmed in \cite{Vorst-char-p}). In \cite{K-DB}, two extensions of Vorst's conjecture were established. First, it was shown that $K_2$-regular affine varieties are normal. Second, it was shown that for local complete intersections of minimal embedding codimension $r$ (see Definition \ref{def:minimal-embedding-codim}), regularity is implied by $K_m$-regularity for $m> \frac{d-r}{2}+1$. Our theorem may be viewed as a further strengthening of Vorst’s conjecture along these lines.

\vspace \medskipamount

The proof of Theorem~\ref{intro-thm:char-of-K-regularity} and Theorem \ref{intro-thm:regularity}, along with other applications discussed later in the paper, combines fundamental results of Corti\~nas, Haesemeyer, Schlichting, Walker and Weibel \cites{K-Weibel,K-Vorst,K-Bass}—now standard in the study of $K$-regularity—with a new ingredient: the theory of \textit{higher Du Bois singularities}.

Recently, developments in Hodge theory have led to considerable interest in the study of higher analogs of the classical notion of Du Bois singularities for complex algebraic varieties. Following \cites{JKSY,MOPW,MP-lci}, a local complete intersection $X$ is said to have \textit{$m$-Du Bois singularities} if the natural maps
\[\Omega_{X/\CC}^p \to \DB_{X/\CC}^p,\]
from the sheaf of algebraic differential forms to the Du Bois complexes (which are natural objects in the bounded derived category $D^b_{\rm coh}(X)$ of coherent sheaves on $X$, arising from Deligne's mixed Hodge theory) are isomorphisms for all $p\le m$.

For varieties that are not local complete intersections, there is currently no known example satisfying $\Omega_{X/\CC}^1\simeq \DB_{X/\CC}^1$. For this reason, one works with a less restrictive notion for general complex algebraic varieties. Following \cite{SVV}, we say $X$ has \textit{pre-$m$-Du Bois singularities} if 
\[\H^i\DB_{X/\CC}^p = 0\text{ for all $i>0$ and $p\le m$},\]
where we write $\H^i \F$ for the $i$-th cohomology sheaf of an object $\F\in D^b_{\rm coh}(X)$. Several other conditions are imposed in \textit{loc. cit.} for the definition of $m$-Du Bois singularities in general. See Section \ref{sec:DB-definition} for the natural extensions of these notions over more general base fields. 

\vspace \medskipamount
The following theorem relates higher Du Bois singularities to $K$-regularity. A related result can be found in \cite{K-DB}, where it is shown that $K_m$-regularity implies the maps $\Omega_X^p\to \DB_X^p$ are isomorphisms for $p\le m-1$.

\begin{intro-theorem}\label{intro-thm:relation-lci}
    Let $X$ be a variety of dimension $d$, and denote $s=\dim X_{\sing}$. Then
    \begin{enumerate}
        \item If $X$ has Du Bois singularities, then it is $K_{-d+1}$-regular;
        \item If $X$ is pre-$m$-Du Bois, then it is $K_{-t}$-regular for $t\ge \max\{s+1, d-2m-2+s\}$; if $X$ is in addition affine, then the bound can be improved to $t\ge \max\{1, d-2m-2\}$.
        \item If $X$ is affine, and either of the following holds:
        \begin{itemize}
            \item $X$ is a local complete intersection;
            \item $\depth \H^0\DB_X^p\ge d-p$ for all $p\ge 0$.
        \end{itemize}
        If $X$ is $K_{-d+2m+1+s}$-regular, then it is pre-$m$-Du Bois.
    \end{enumerate}
\end{intro-theorem}

This theorem, combined with results from the theory of higher Du Bois singularities, most notably the characterization via the minimal exponent \cites{MOPW,JKSY} and the codimension bound on the singular locus (Theorem \ref{thm:MP-codim-bound}), serves as the main technical input in the proofs of Theorems \ref{intro-thm:char-of-K-regularity} and \ref{intro-thm:regularity}. In addition, parts (1) and (2) produce new classes of $K_m$-regular varieties, both in the local complete intersection setting and beyond; see Example \ref{ex: non-lci-K-regular}.

\vskip \medskipamount
We conclude with two further applications of Hodge-theoretic methods to $K$-regularity. First, following \cites{K-Bass,K-Bass-negative}, we use results on Du Bois invariants (Proposition \ref{prop:DB-invariants}) to obtain new examples in which Bass' question about whether $\AA^1$-invariance of the algebraic $K$-groups implies $\AA^2$-invariance has either a positive or a negative answer; see Section \ref{sec:DB-table}.

Second, we use the degeneration of the Hodge-to-de Rham spectral sequence to characterize $K$-regularity for projective varieties:

\begin{intro-theorem}\label{intro-thm:CHWW-projective}
    Let $X$ be a complex projective variety. Then $X$ is $K_m$-regular if and only if the natural maps
    \[\HH^i(X,\LL_{X/\CC}^p)\xto{\sim} \HH^i(X, \DB_{X/\CC}^p)\]
    are isomorphisms for all $i-p\ge -m+1$. Here, we denote by $\HH^i(X, -)$ the hypercohomology functor on $X$, and by $\LL^p_{X/\CC}$ the $p$-th derived wedge product of the cotangent complex; see Section \ref{subsec:derived-dR}.
\end{intro-theorem}

This leads to examples showing that local $K_m$-regularity need not imply $K_m$-regularity; see Example \ref{ex:locally-K-regular-does-not-imply-K-regular}.

\vskip \medskipamount

\noindent\textbf{Acknowledgement.} I would like to express my sincere gratitude to Mihnea Popa for his constant support and encouragement during the preparation of this paper, and for raising a question in the AIM workshop on higher Du Bois and higher rational singularities, which led to Section \ref{sec:DB-table}. I am extremely grateful to Elden Elmanto, who introduced me to the work of \cite{K-Weibel}, and to Dori Bejleri and Anh Duc Vo for the discussions that shaped the material presented in Section \ref{sec:K-reg-for-proj-var}. I am also thankful to Anh Duc Vo and Xinyu Fang for pointing out mistakes in an earlier draft of this paper. I am indebted to Bradley Dirks, from whom I learned about Theorem \ref{thm:codim-bound}; and to Ming Hao Quek, who helped me work out the examples in Proposition \ref{prop:DB-invariants}. I thank Toni Annala, Haoyang Liu, Sung Gi Park and Noah Riggenbach for helpful conversations, and Chuck Weibel for comments on an earlier version of this paper. I am also grateful to an anonymous referee for a careful reading of the paper and for many helpful comments.

\section{Preliminaries}
\subsection{The cdh-differentials and the Du Bois complexes}
Let $X$ be a Noetherian scheme over a field $F$. The \textit{cdh differentials} of $X$ are defined as the sheafification of the K\"ahler differentials $\Omega_{X/F}^p$ in the cdh-topology introduced by Suslin and Voevodsky \cite{SV}. More precisely, 
\[\Omega_{cdh, X/F}^p := \bR a_*a^*\Omega_{X/F}^p,\]
where $a: X_{cdh}\to X_{zar}$ is the natural change-of-topology morphism from the cdh site to the Zariski site.

On the other hand, for a complex algebraic variety (i.e. when $F=\CC$), Du Bois \cite{DB} introduced the notion of \textit{Du Bois complexes} (or \textit{filtered de Rham complexes}) following ideas of Deligne. The same construction applies to define the Du Bois complexes of any quasi-excellent Noetherian scheme $X$ over a field $F$ of characteristic zero. Indeed, by Hironaka's theorems, the assumptions on $X$ guarantee the existence of a hyperresolution $\epsilon_{\bullet}: X_{\bullet}\to X$. We then define the \textit{p-th Du Bois complex} of $X$ over $F$ as
\[\DB_{X/F}^p:= \bR \epsilon_{\bullet *} \Omega^p_{X_{\bullet}/F}.\]

We refer the readers to \cite{GNPP}*{Chapter V} and \cite{PS}*{Chapter 7.3} for the construction of hyperresolutions, and for the definition of derived pushforward from a hyperresolution. Being natural objects arising from Deligne's mixed Hodge theory, the Du Bois complexes are mostly studied for complex algebraic varieties. In this case, the Du Bois complexes $\DB_{X/\CC}^p$ in our notations are usually denoted simply as $\DB_X^p$ in the literature.

The relationship between the Du Bois complexes and the cdh-differentials is observed in \cite{K-Bass-negative}*{Lemma 2.1}:

\begin{lemma}\label{lemma:DB-cdh}
    Let $X$ be a finite type $F$-scheme, and let $k\subset F$ be a subfield. Then for $p\ge 0$, there is an isomorphism
    \[\DB_{X/k}^p\simeq \Omega_{cdh, X/k}^p.\]
    \begin{proof}
        When $k=F$, this is exactly \cite{K-Bass-negative}*{Lemma 2.1}. For $k\subset F$, the result follows by the same proof, using \cite{K-Vorst}*{Corollary 2.5} and \cite{SV}*{Lemma 12.1}.
    \end{proof}
\end{lemma}

For an $F$-scheme $X$, we denote by $\DB_{X/F}^{\le p}$ the (derived) cdh-sheafification of the following complex of quasi-coherent sheaves:
\[\O_X\to \Omega_{X/F}^1 \to \cdots \to \Omega_{X/F}^p\]  
supported in cohomological degrees $0,1,\cdots, p$. Explicitly, $\Omega_{cdh, X/F}^{\le p} := \bR a_*a^*\Omega_{X/F}^{\le p}$, where $a: X_{cdh}\to X_{zar}$ is the natural change-of-topology morphism from the cdh site to the Zariski site.

When $X$ is of finite type over $F$, this represents an object in the derived category of differential complexes on $X$. When $F=\CC$, it follows from Lemma \ref{lemma:DB-cdh} that for each $p\ge 0$,
\[\DB_{X/\CC}^{\le p} = \DB_{X/\CC}^{\bullet}/F^{p+1}\]
agrees with the quotient of the filtered de Rham complex by the Hodge filtration. It is immediate from the definition that we have the following exact triangle for each $p\ge 0$:
\begin{equation}\label{eqn:le-p}
    \DB_{X/F}^p[-p]\to \DB_{X/F}^{\le p} \to \DB_{X/F}^{\le p-1}\xto{+1}.
\end{equation}

In the proof of \cite{K-Weibel}*{Theorem 6.1} and \cite{K-Vorst}*{Proposition 2.6}, the following useful properties of the objects $\DB_{X/F}^p$ and $\DB_{X/F}^{\le p}$ are established:
\begin{proposition}\label{prop:DB-complex-support}
    Let $X$ be a separated, finite type $F$-scheme, and let $k \subset F$ be a subfield. Then for each $p\ge 0$,
    \begin{enumerate}
        \item $\dim \supp \H^q \DB_{X/k}^p \le \dim X - q - 1$ if $q>0$; in particular, $\H^i\DB_{X/k}^p=0$ for all $p$ and all $i\ge d$;
        \item $\HH^q(X,\DB_{X/k}^p)=0$ for all $q>\dim X$;
        \item $\HH^q(X,\DB_{X/k}^{\le p})=0$ for all $q>\dim X+p$.
    \end{enumerate}
    \begin{proof}
        (1) is stated in the proof of \cite{K-Weibel}*{Theorem 6.1} when $p=0$ (in which case $\DB_{X/k}^0=\DB_{X/F}^0$). The proof for $p\ge 1$ is identical when $k=F$, and the case for general $k\subset F$ follows from Lemma \ref{lemma: DB-change-of-field}(1). (2) follows from (1) and the spectral sequence
        \[E_1^{ij}=H^i(X, \H^j\DB_{X/F}^p)\implies \HH^{i+j}(X,\DB_{X/F}^p).\]
        Finally, (3) follows from (2) by induction using the exact triangle (\ref{eqn:le-p}).
    \end{proof}
\end{proposition}

\subsection{Derived de Rham complex and the cotangent complex}\label{subsec:derived-dR} 
Following Bhatt \cite{Bhatt-derived}*{Construction 4.1}, for an $F$-algebra $A$, we define the \textit{(Hodge-completed) derived de Rham complex of $A$ over $F$} as
\[\dR_{A/F}:=\lim_{p\in \NN^{op}} L\Omega_{A/F}^{\le p},\]
where $L\Omega_{A/F}^{\le p}:=\mathrm{Tot}(\Omega_{P/F}^{\le p}\otimes_P A)$ is the single complex associated to the double complex $\Omega_{P}^{\le p}\otimes_P A$, where $P\to A$ is a polynomial $F$-algebra resolution of $A$. It can be shown that the objects $L\Omega_{A/F}^{\le p}$ and $\dR_{A/F}$ are independent of the choice of the polynomial resolution $P\to A$ (up to homotopy), and that the construction globalizes to give definitions of $L\Omega_{X/F}^{\le p}$ and $\dR_{X/F}$ for any $F$-scheme $X$.

By construction, there is a natural \textit{derived filtration} $\Fil_H$ on $\widehat{\rm dR}_{X/F}$, with respect to which the graded pieces are given by
\[\gr^p_H \dR_{X/F} = \LL^p_{X/F}[-p]\]
where $\LL^p_{X/F}$ is the $p$-th derived wedge product of the cotangent complex, introduced by Illusie \cite{Illusie-cotangent-complex}. 

As explained in \cite{Bhatt-derived}*{section 5}, for a finite type $F$-scheme $X$, there is a map
\[\varphi:\dR_{X/F}\to \DB_{X/F}^{\bullet}\] 
preserving the derived filtration on $\dR_{X/F}$ and the Hodge filtration on $\DB_{X/F}^{\bullet}$. Taking associated graded pieces with respect to the two filtrations, we obtain a natural map
\[\LL_{X/F}^p\to \DB_{X/F}^p\] 
for each $p$. It is a natural question when these maps are isomorphisms. For local complete intersections over $F=\CC$, this question is studied in the theory of higher Du Bois singularities. When $X$ is not a local complete intersection, this map is rarely an isomorphism unless $p=0$; see Remark \ref{rmk:no-non-lci-example}.

\subsection{Background in higher Du Bois singularities}\label{sec:DB-definition}
In the literature, higher Du Bois singularities are defined and studied for complex algebraic varieties. For the purpose of relating to $K$-regularity, we follow the notation in the previous subsections, and work with these singularities over more general base fields.

\iffalse Using notations from the previous subsection, we consider this notion for any $F$-scheme. There are two reasons such generality is desirable: 
\begin{enumerate}
    \item The $K$-theory of a complex algebraic variety $X$ is related to $\DB_{X/\QQ}^p$, whereas higher Du Bois singularities is defined using $\DB_{X/\CC}^p$. To relate these notions, it is useful to have a notion of \textit{higher Du Bois singularities over $\QQ$}.
    \item For a family $f: \X\to B$ of algebraic varieties, we can talk about the family $\X$ having \textit{higher Du Bois singularities over $B$}. Studying how this notion relates to the fibers having higher Du Bois singularities may have interesting applications to moduli problems.
\end{enumerate} \fi

Let $X$ be a local complete intersection over $F$, and $k\subset F$ a subfield. Following \cites{MOPW, JKSY} (with the obvious replacement of $\CC$ by $F$), we say $X$ has \textit{$m$-Du Bois singularities over $k$} if the natural maps $\Omega_{X/k}^p\to \DB_{X/k}^p$ are isomorphisms for all $p\le m$, where $\Omega_{X/k}^p$ is the (underived) $p$-th wedge product of the K\"ahler differential $\Omega_{X/k}^1$ relative to $k$. One important property of $m$-Du Bois singularities is the following codimension bound on the singular locus, obtained by combining \cite{MP-LC}*{Theorem F} and \cite{MP-LC}*{Corollary 9.26}.
\begin{theorem}[\Mustata-Popa]\label{thm:MP-codim-bound}
    Let $X$ be a local complete intersection. If $X$ has $m$-Du Bois singularities, then
    \[\codim X_{\sing}\ge 2m+1.\]
\end{theorem}

As a consequence, for local complete intersections one can replace $\Omega_{X/k}^p$ in the definition of $m$-Du Bois singularities by the more natural object $\LL_{X/k}^p$, the $p$-th derived wedge product of the cotangent complex of $X$ over $k$:

\begin{lemma}\label{lemma:lci-DB-def}
    Let $X$ be a local complete intersection over $F$, and $k\subset F$ a subfield. Then the following are equivalent:
    \begin{enumerate}
        \item $X$ has $m$-Du Bois singularities over $k$;
        \item $X$ has $m$-Du Bois singularities over $F$;
        \item $X$ has $m$-Du Bois singularities over $F'$, for any field extension $F\subset F'$;
        \item The natural maps $\LL_{X/k}^p\to \DB_{X/k}^p$ are isomorphisms for all $p\le m$;
        \item The natural maps $\LL_{X/F}^p\to \DB_{X/F}^p$ are isomorphisms for all $p\le m$;
        \item The natural maps $\LL_{X\times_F F'/F'}^p\to \DB_{X\times_F F'/F'}^p$ are isomorphisms for all $p\le m$, and any field extension $F\subset F'$.
    \end{enumerate}
    \begin{proof}
        The equivalences $(1)\iff (2)$ and $(4)\iff (5)$ follow from Lemma \ref{lemma: DB-change-of-field}(1), and the equivalences $(2)\iff (3)$ and $(5)\iff (6)$ follows from Lemma \ref{lemma: DB-change-of-field}(2). 
        
        We will show $(2)\iff (5)$: using Lemma \ref{lemma: DB-change-of-field}, we reduce to the case $k=\CC$\footnote{More precisely, given a variety $X/k$ with $m$-Du Bois singularities over $k$, we can write $X\simeq X(k_0)\times_{k_0} k$ for some extension field $k_0$ of $\QQ$ of finite transcendence degree. Then, the field $k_0$ embeds into $\CC$. By the equivalences established above, $X$ has $m$-Du Bois singularities if and only if the complex algebraic variety $X(k_0)\times_{k_0} \CC$ has $m$-Du Bois singularities. If we can show this is equivalent to $\LL^p_{X(k_0)\times_{k_0} \CC/\CC}\to \DB^p_{X(k_0)\times_{k_0} \CC/\CC}$ are isomorphisms for $p\le m$, applying the equivalence $(5)\iff (6)$ then gives the result.}. For ease of notation, we then write $\LL_{X/\CC}^p$ and $\DB_{X/\CC}^p$ simply as $\LL_{X}^p$ and $\DB_{X}^p$.

    If the natural maps $\LL_X^p\to \DB_X^p$ are isomorphisms for all $p\le m$, taking cohomology in degree $\ge 0$ shows that $X$ has $m$-Du Bois singularities. 

    For the converse, suppose $X$ is a local complete intersection subvariety of a smooth variety $Y$, and that $X$ has $m$-Du Bois singularities. Denote by $\I=\I_{X/Y}$ the ideal sheaf defining $X$ in $Y$. Then, the cotangent complex $\LL^p_X$ can be represented by
    \[\I^p/\I^{p+1}\to \I^{p-1}/\I^p \otimes \Omega_Y^1|_X \to \cdots \to \I/\I^2 \otimes \Omega_Y^{p-1}|_X\to \Omega_Y^p|_X,\]
    supported in cohomological degrees $-p, -p+1, \cdots, -1, 0$.

    This is the Eagon-Northcott complex associated to (the dual of) the map $\I/\I^2\to \Omega_Y^1|_X$. As explained in \cite{MP-LC} before the proof of Theorem F\footnote{In loc. cit., this complex shows up as the Grothendieck dual of $\Gr^E_{-p} \DR(\H^r_X \O_Y)$, the graded piece of the Hodge module $\H^r_X \O_Y$ with respect to the Ext filtration.}, by the depth-sensitivity of the Eagon-Northcott complex, it is a resolution of $\Omega_X^p$ when $\codim X_{\sing} \ge p$. This is satisfied when $X$ has $m$-Du Bois singularities for $m\ge 1$ (see Theorem \ref{thm:MP-codim-bound}). Since the Lemma is clear when $m=0$ (in which case $\LL_X^0 = \O_X = \Omega_X^0$), this completes the proof.
\end{proof}
\end{lemma}

In view of the equivalence $(1)\iff (2)$, we will sometimes simply say $X$ has \textit{$m$-Du Bois singularities} when the base field causes no confusion.    
    
When $X$ is not necessarily a local complete intersection, as explained in \cite{SVV}, the conditions $\Omega_{X/F}^1\xto{\sim} \DB_{X/F}^1$ and $\LL_{X/F}^1\xto{\sim} \DB_{X/F}^1$ are too restrictive due to the bad behaviors of the K\"ahler differentials. 

\begin{remark}\label{rmk:no-non-lci-example}
    For a variety $X$ over $F$, the sheaf $\H^0\DB_{X/F}^1$ (even $\H^0\DB_{X/F}^p$ for all $p\ge 1$) is
    \begin{itemize}
        \item torsion-free by \cite{h-differential}*{Remark 3.8};
        \item reflexive when $X$ has rational singularities, by \cite{KS-extending}*{Corollary 1.2} and the identification of $\H^0\DB_{X/F}^1$ with the $h$-differential in \cite{h-differential}. 
    \end{itemize} 
    On the other hand, when $X$ is not a local complete intersection, the K\"ahler differentials typically have torsion, and are not reflexive even when $X$ has mild singularities; see \cite{Kahler-diff-torsion}. For such varieties, one does not have $\Omega_{X/F}^1\simeq \H^0\DB_{X/F}^1$. 

    As a consequence, when $X$ is not a local complete intersection, the maps $\Omega_{X/F}^1\to \DB_{X/F}^1$ and $\LL_{X/F}^1\to \DB_{X/F}^1$ need not be isomorphisms even when $X$ has mild singularities. 
\end{remark}

In fact, currently there is no known example of a variety satisfying $\Omega_X^1\simeq \H^0\DB_X^1$ that is not a local complete intersection. Therefore, for arbitrary varieties, we work with the following definition. The equivalences follow from Lemma \ref{lemma: DB-change-of-field}.
\begin{definition}
    Let $X$ be a variety over $F$, and $k\subset F$ a subfield. We say that $X$ has \textit{pre-$m$-Du Bois singularities} if one of the following equivalent conditions hold:
    \begin{enumerate}
        \item $\H^i\DB_{X/F}^p=0$ for all $p\le m$ and $i>0$;
        \item $\H^i\DB_{X/k}^p=0$ for all $p\le m$ and $i>0$;
        \item $\H^i\DB_{X\times_F F'/F'}^p=0$ for all $p\le m$ and $i>0$, for some (any) field extension $F'$ of $F$.
    \end{enumerate}
\end{definition}

\begin{remark}
    Criteria for various classes of varieties (that are typically not local complete intersections) to have pre-$m$-Du Bois singularities have been worked out. See \cite{GNPP}*{V Theorem 4.6} for toric varieties, \cites{K-cones,PSh} for cones over projective varieties, and \cite{secant} for secant varieties. 
    
    For local complete intersections, one can characterize $m$-Du Bois singularities using the minimal exponent \cites{JKSY,MOPW,CDMO,CDM}, which is an invariant that refines the log canonical threshold.
    
    On the other hand, necessary and sufficient criteria for an arbitrary variety to have pre-$m$-Du Bois singularities are not known at this moment.
\end{remark}

Following \cite{SVV}, we say $X$ has \textit{$m$-Du Bois singularities} if it is seminormal, pre-$m$-Du Bois, the sheaves $\H^0\DB_{X/F}^p$ are reflexive for all $p\le m$, and that $\codim X_{\sing}\ge 2m+1$. As explained in \cite{SVV}*{Section 5}, this new definition agrees with the well-studied notion in the case of local complete intersections, and includes more examples in general. Varieties with $m$-Du Bois singularities in this sense provide examples that are $K_{-d+m+1}$-regular; see Example \ref{ex: non-lci-K-regular}.

\subsection{\texorpdfstring{Background in $K$-regularity}{Background in K-regularity}}
Let $X$ be a scheme. We say $X$ is \textit{$K_m$-regular} if the natural maps
\[K_m(X)\to K_m(X\times \AA^r)\]
between algebraic $K$-groups are isomorphisms for all $r\ge 0$. It is a fact that if $X$ is regular, then it is $K_m$-regular for all $m$. Also, $K_m$-regularity implies $K_{m-1}$-regularity for all $m$; see \cite{Vorst} and \cite{Dayton-Weibel}*{Corollary 4.4}. 

In a sequence of foundational papers \cites{K-Weibel, K-Bass, K-Vorst} in this area, it was discovered that $K_m$-regularity of $X$ is controlled by the K\"ahler differentials, the cdh-differentials and the Hochschild homology of $X$. Following \textit{loc. cit.}, we denote by
\[N K_m(X) = \coker(K_m(X)\to K_m(X\times \AA^1))\]
the $NK$-groups of Bass. For $r\ge 2$, the groups $N^r K(X)$ are inductively defined by
\[N^r K_m(X):= N(N^{r-1} K_m)(X) := \coker(N^{r-1} K_m(X)\to N^{r-1} K_m(X\times \AA^1))\]

By \cite{K-Bass}*{Theorem 0.1}, we have
\begin{align*}
   N^r K_m(X) = 0 &\iff NK_m(X) = NK_{m-1}(X) = \cdots = NK_{m-r+1}(X) = 0 \\
   &\iff K_m(X) \xto{\sim} K_m(X\times \AA^r) \text{ is an isomorphism}.
\end{align*}

This reduces the study of $K$-regularity to that of the $NK$-groups. Another theorem in \cite{K-Bass} expresses $NK$, or more precisely, $NK^{(i)}$, the $i$-th eigenspace of the $\lambda$-operation on $NK$, in terms of the K\"ahler differentials, the cdh differentials and the Hochschild homology of $X$. We rephrase the statement with the notation of Du Bois complexes, using Lemma \ref{lemma:DB-cdh}.

\begin{theorem}\cite{K-Bass}*{Theorem 0.4}\label{thm:CHWW-main}
Let $X=\Spec R$ be an affine scheme over $\QQ$. Then
    \[NK_n^{(i)}(X) \cong TK_n^{(i)}(X)\otimes_{\QQ} t\QQ[t]\]
where $t$ is the coordinate on $\AA^1$ in the definition of $NK$, and
\[TK_n^{(i)}(X) \simeq 
\begin{cases}
    \H^{i-n-1}\DB_{X/\QQ}^{i-1} \quad &n\le i-2 \\
    \coker(\Omega_{X/\QQ}^i\to \H^0\DB_{X/\QQ}^i) \quad &n = i-1 \\

    \ker(\Omega_{X/\QQ}^i\to \H^0\DB_{X/\QQ}^i) \quad &n=i \\
    HH^{(i-1)}_{n-1}(X) \quad &n\ge i+1
\end{cases}\]
\end{theorem}

\begin{remark}\label{rmk:CHWW-main}
    Since $HH_n^{(i)}(X)\simeq \H_{n-i} \LL^i \simeq \H^{i-n} \LL^i$, this theorem can be rephrased in the following way: Let $C^i$ be the cone of the natural map $\LL^i_{X/\QQ} \to \DB^i_{X/\QQ}$. Then the
    $TK_n^{(i)}(X)$'s are determined by the cohomologies of $C^i$:
    \[TK^{(i+1)}_{i-j}(X) \simeq  \H^j C^i.\]  
    As a consequence, $X$ is $K_m$-regular if and only if
    \[\H^j C^i =0 \quad \text{ for all } j-i\ge -m.\]
 
    Thus, in view of Lemma \ref{lemma:lci-DB-def}, the notion of $K$-regularity is naturally related to higher singularities, at least for local complete intersections.
\end{remark}

To study the $K$-regularity for varieties that are not necessarily affine, it is convenient to work with a local version of Bass' $NK$-groups, denoted $\NK_m$, and defined as the sheafification of
\[U\mapsto NK_m(\Gamma(U))\]
on the \'etale site of $X$ in \cite{LocalNK}. As shown in \textit{loc. cit.}, we have
\[H^i_{\text{\'et}}(X, \NK_m) = H^i(X, \NK_m),\]
where $H^i$ denotes the cohomology taken in the Zariski topology. When $X$ is affine, 
\[H^i_{\text{\'et}}(X, \NK_m) = \begin{cases}
NK_m(X) \quad &\text{if $i=0$}, \\
0 \quad &\text{if $i>0$}.
\end{cases}\]
We have a local-to-global spectral sequence relating local and global $NK$-groups
\[E_2^{pq} = H^p(X, \NK_q(X))\implies NK_{q-p}(X).\]
In view of these results, Theorem \ref{thm:CHWW-main} can be stated in this local setting as follows:
\begin{theorem}[\cite{K-Bass}]\label{thm:CHWW-local}
    Let $X$ be any scheme over $\QQ$. Then 
    \[\NK_m(X) \simeq \bigoplus_{j-i=-m} \H^j C^i \otimes t\QQ[t],\]
    where $C^i = \Cone(\LL^i_{X/\QQ}\to \DB_{X/\QQ}^i)$. In particular, $X$ is \textit{locally $K_m$-regular}, i.e. $\NK_q(X)=0$ for all $q\le m$, if and only if 
    \[\H^j C^i = 0 \text{ for } j-i \ge -m.\]
\end{theorem}

\iffalse
\begin{remark}
    The natural map $\LL^i_{X/\QQ}\to \DB_{X/\QQ}^i$ can be viewed as the associated graded of the map
    \[\widehat{\text{dR}}_{X/\QQ} \to \DB_{X/\QQ}^{\bullet}\]
    where $\widehat{\text{dR}}_{X/\QQ}$ is the derived de Rham cohomology equipped with a derived filtration \cite{Bhatt-derived}*{Section 5}, and $\DB_{X/\QQ}^{\bullet}$ is the Du Bois complex defined relative to $\QQ$, equipped with the Hodge filtration. In the lci case, these filtrations are related by duality to the Ext and Hodge filtrations \cite{MP-LC} on the local cohomology $\H^r_X \O_Y$ for a codimension $r$ embedding of $X$ in some smooth variety $Y$. This further explains the relation of $K$-regularity with higher Du Bois singularities in the lci case.    
\end{remark}
\fi

\subsection{Base change results}
We saw that the theory of higher Du Bois singularities concerns the study of the Du Bois complex $\DB_{X/F}^p$ for a variety over $F$, whereas the $K$-regularity of $X$ is related to $\LL_{X/\QQ}^p$ and $\DB_{X/\QQ}^p$. Thus, to relate the two theories we will need base change results for the Du Bois complex and the cotangent complex over various fields, which we collect below.

\begin{lemma}\label{lemma: DB-change-of-field}
    Let $X$ be a variety over $F$.
    \begin{enumerate}
        \item (\cite{HH-base-change}*{4.3a}, \cite{K-Vorst}*{Lemma 4.1, 4.2})
        If $k\subset F$ is a subfield, then we have spectral sequences
        \[E_1^{ij}= \Omega^i_{F/k}\otimes_F \H^{i+j}\LL_{X/F}^{p-i}\implies \H^{i+j}\LL_{X/k}^p\]
        \[E_1^{ij}= \Omega^i_{F/k}\otimes_F \H^{i+j}\DB_{X/F}^{p-i}\implies \H^{i+j}\DB_{X/k}^p\]
        \[E_1^{ij}=\Omega^i_{F/k}\otimes \HH^j(X, L\Omega_{X/F}^{\le p-i})\implies \HH^{i+j}(X, L\Omega_{X/k}^{\le p})\]
        \[E_1^{ij}=\Omega^i_{F/k}\otimes \HH^j(X, \DB_{X/F}^{\le p-i})\implies \HH^{i+j}(X, \DB_{X/k}^{\le p})\]

        \item If $F \subset F'$ is a field extension, then we have isomorphisms
        \[\LL_{X\times_F F'/F'}^p \simeq \LL_{X/F}^p \otimes_F F'.\]
        \[\DB_{X\times_F F'/F'}^p \simeq \DB_{X/F}^p \otimes_F F'.\]
    \end{enumerate}
    \begin{proof}
        (1) The first spectral sequence follows from \cite{HH-base-change}*{4.3a}, using that $HH_n^{(i)}(X)\simeq \H_{n-i} \LL^i \simeq \H^{i-n} \LL^i$. The second one is exactly \cite{K-Vorst}*{Lemma 4.2}, by the comparison between cdh differentials and the Du Bois complexes (see Lemma \ref{lemma:DB-cdh}). The last two spectral sequences are proved along the same lines as in \cite{HH-base-change}*{4.3a} and \cite{K-Vorst}*{Lemma 4.2}, using instead the filtration $\G$ on $\Omega_{X/\QQ}^{\le p}$ given by
        \[\G^i/\G^{i+1} \simeq \Omega_{\CC/\QQ}^i \otimes_{\CC} \Omega_{X/\CC}^{\le p-i}[-i]\]
        when $X$ is smooth.
        
        (2)  
        For K\"ahler differentials, we have isomorphism
        \[\Omega_{X\times_k k'/k'}^p \simeq \Omega_{X/k}^p \otimes_k k'.\]
        This gives the two desired isomorphisms, taking left derived functor for the first, and cdh sheafification for the second.
    \end{proof}
\end{lemma}

\section{Vanishing theorems for the Du Bois complexes}\label{sec:vanishing-thm}
In this section, we collect vanishing theorems for the Du Bois complexes, with focus on the effect of the change of base fields on the statements. These will be applied later to study $K$-regularity.

The first general result is due to Steenbrink when $F=\CC$. The general case follows from Lemma \ref{lemma: DB-change-of-field}.
\begin{theorem}[\cite{Steenbrink-vanishing}]\label{thm:Steenbrink-vanishing}
    Let $X$ be a variety over $F$. Then
    \[\H^q \DB_{X/F}^p = 0\]
    for $p+q>\dim X$.
\end{theorem}

In view of Lemma \ref{lemma: DB-change-of-field}(1), Steenbrink's vanishing theorem need not hold for Du Bois complexes over a subfield $k\subset F$. Instead, we have the following weaker statement:
\begin{proposition}\label{prop:DB-vanishing-over-subfield}
    Let $X$ be a variety over $F$, and $k\subset F$ a subfield. If $X$ has pre-$m$-Du Bois singularities, then \[\H^q \DB_{X/k}^p = 0\] for all $q\ge \dim X-m-1$ and $q>0$.
\begin{proof}
    It suffices to prove the result when $p\ge m+1$. By \cite{PSV}*{Proposition C} and Steenbrink vanishing (and reduction to the case $F=\CC$ as usual), we have that $\H^q \DB_{X/F}^p=0$ for $q\ge \dim X -p\ge \dim X-m-1$. It then follows from the spectral sequence in Lemma \ref{lemma: DB-change-of-field}(1) that $\H^q\DB^p_{X/k}=0$ for $q\ge \dim X-m-1$.
\end{proof}
\end{proposition}

Complementary to Steenbrink's theorem, we have vanishing results for $\H^q\DB_{X/k}^p$ when $p+q$ is small:
\begin{theorem}[\cite{MP-lci}]\label{thm:MP-vanishing}
    Let $X$ be a local complete intersection over $F$, and $k\subset F$ a subfield. Denote by $X_{\sing}$ the singular locus of $X$. Then for all $p\ge 0$, we have
    \[\H^q \DB_{X/k}^p = 0\]
    for $1\le q< \codim X_{\sing}-p-1$. 
    \begin{proof}
        When $F=\CC$, this is \cite{MP-lci}*{Corollary 13.9}. The result for $F$ follows from Lemma \ref{lemma: DB-change-of-field}(2), and the statement for a subfield $k\subset F$ follows by analyzing the spectral sequence in Lemma \ref{lemma: DB-change-of-field}(1).
    \end{proof}
\end{theorem}

Beyond local complete intersections, we have:
\begin{theorem}[\cites{PSV,Kovacs-inj}]\label{thm:PSV-vanishing}
    Let $X$ be a variety over $F$, and $k\subset F$ be a subfield. Then
    \begin{enumerate}
        \item $\H^q\DB_{X/F}^0 = \H^q\DB_{X/k}^0=0$ for $1\le q < \depth \O_X-\dim X_{\sing}-1$.
        \item If $X$ is pre-$(p-1)$-Du Bois with $\dim X_{\sing} = s$, then
        \[\H^q \DB_{X/F}^p = 0\] 
        for $1\le q< \depth \H^0\DB_{X/F}^p-s-1$. In particular, if $\depth \H^0\DB_{X/F}^p\ge \dim X-p$ for all $p\ge 0$, then
        \[\H^q\DB_{X/k}^p = 0\]
        for $1\le q < \dim X - p -s - 1$, and any subfield $k\subset F$.
        \end{enumerate}
    \begin{proof}
        Part (1) follows from \cite{PSV}*{Corollary B}, using Lemma \ref{lemma: DB-change-of-field}(2) and the fact $\DB_{X/F}^0=\DB_{X/k}^0$. Part (2) follows from \cite{PSV}*{Theorem A}, \cite{Kovacs-inj}*{Theorem 1.1} and Lemma \ref{lemma: DB-change-of-field}. 
    \end{proof}
\end{theorem}

\section{\texorpdfstring{Higher Du Bois singularities and $K$-regularity}{Higher Du Bois singularities and K-regularity}}

We start by establishing a general relationship between higher Du Bois singularities and $K$-regularity. Later, we will apply vanishing theorems for the Du Bois complexes to improve these results under additional hypothesis.
\begin{proposition}($\implies$ Theorem \ref{intro-thm:relation-lci}(1))\label{prop:relation-general}
    Let $X$ be a variety of dimension $d$. 
    \begin{enumerate}
        \item If $X$ has $m$-Du Bois singularities, then it is $K_{-d+m+1}$-regular;
        \item Suppose $X$ is either affine, or has isolated singularities. If $X$ is $K_m$-regular, then it is pre-$(m+1)$-Du Bois.
    \end{enumerate}
\begin{proof}
(1) Since $X$ has $m$-Du Bois singularities in the sense of \cite{SVV}, the following conditions are satisfied:
\begin{itemize}
            \item $X$ is seminormal;
            \item $X$ has pre-$m$-Du Bois singularities;
            \item $\codim X_{\sing}\ge m+2$ if $m\ge 1$.
        \end{itemize}

To show $X$ is $K_{-d+m+1}$-regular, by the local-to-global spectral sequence
\[E_2^{pq} = H^p(X, \N\K_q(X)) \implies NK_{q-p}(X),\]
it suffices to show 
\[H^p(X, \N\K_q(X))=0 \text{ for all $q-p \le -d+m+1$}.\]
By Theorem \ref{thm:CHWW-local}, this is equivalent to showing
\[H^p(X, \H^s C^t) =0 \text{ for $s-t=-q$ and $q-p\le -d+m+1$},\]
where $C^t$ is the cone of the natural map $\LL_{X/\QQ}^t\to \DB_{X/\QQ}^t$. 

We verify this through a case-by-case analysis:
\begin{itemize}
    \item If $s>0$ and $t\le m$, we have $\H^s C^t = 0$ because $X$ is pre-$m$-Du Bois;
    \item If $s>0$ and $t\ge m+1$, then $p\ge q+d-m-1 = t-s+d-m-1\ge d-s$. By Proposition \ref{prop:DB-complex-support},
    \[\dim \supp \H^s C^t = \dim \supp \H^s \DB_{X/\QQ}^t\le d-s-1,\]
    Thus, we have $H^p(\H^s C^t)=0$.
    \item If $s\le 0$, then $q\ge s+q = t\ge 0$, and $p\ge q+d-m-1\ge d-m-1$. By assumption, the sheaves $\H^s C^t$ are supported on $X_{\sing}$, which has dimension $\le d-m-2$ when $m\ge 1$, so we have $H^p(\H^s C^t)=0$. 
    
    \noindent For $m=0$, some further caseworks are needed:
    \begin{itemize}
        \item If $t=0$, then $\H^s C^t = 0$;
        \item If $t\ge 1$, then $p\ge q+d-1 = d-1-s+t\ge d$, so $H^p(\H^s C^t)=0$, because $\H^s C^t$ is supported on $X_{\sing}$, which has dimension $\le d-1$.
    \end{itemize}

\end{itemize}
This completes the proof that $X$ is $K_{-d+m+1}$-regular.

(2) Suppose $X$ is affine and $K_m$-regular. By Theorem \ref{thm:CHWW-main}, we have \[\H^s C^t = 0 \text{ for $s-t\ge -m$}.\]
In particular, taking $s>0$, this implies 
\[\H^s\DB_{X/\QQ}^p=0\]
for all $p\le m+1$ and $s>0$, i.e. that $X$ is pre-$(m+1)$-Du Bois.

Now, suppose $X$ has isolated singularities. Then, the spectral sequence
\[E_2^{pq} = H^p(X,\NK_q(X))\implies NK_{q-p}(X)\]
degenerates, giving that the local $NK$-groups $\NK_q(X)=0$ for $q\le m$. It then follows from the affine case that $X$ is pre-$(m+1)$-Du Bois.
\end{proof}
\end{proposition}

Proposition \ref{prop:relation-general}(1) allows us to construct many examples of $K_{-d+m+1}$-regular varieties:

\begin{example}\label{ex: non-lci-K-regular}
    We indicate some ways to construct varieties with $m$-Du Bois singularities:
    \begin{enumerate}
        \item A codimension $r$ local complete intersection $X$ has $m$-Du Bois singularities if and only if its minimal exponent $\tilde{\alpha}_X\ge m+r$; see \cites{Saito-min-exponent,JKSY, MOPW, CDM}. Playing around with the minimal exponents, one can construct varieties with $m$-Du Bois singularities;
        \item For toric varieties and cones over smooth projective varieties, one can find criteria for them to have $m$-Du Bois singularities in \cite{SVV}*{Section 6-7}.
        \item Criteria for secant varieties to have $m$-Du Bois singularities are worked out in \cite{secant}.
    \end{enumerate}

    By Proposition \ref{prop:relation-general}(1), these provide examples of $K_{-d+m+1}$-regular varieties.
\end{example}

Next, applying vanishing theorems in Section \ref{sec:vanishing-thm}, we obtain the following strengthening of Proposition \ref{prop:relation-general} in the local setting: 

\begin{proposition}[$\implies$ Theorem \ref{intro-thm:relation-lci}(2)]
    \label{prop:DB-implies-neg-regular-affine-case}
    Let $X$ be a $d$-dimensional variety over $F$ with pre-$m$-Du Bois singularities. Then
    \begin{enumerate}
        \item $X$ is locally $K_{-t}$-regular if $t\ge \max\{1, d-2m-2\}$;
        \item If $X$ is seminormal, then it is locally $K_0$-regular if $m\ge \frac{d-2}{2}$;
        \item If $X$ is seminormal, and the natural map $\Omega_{X/F}^1\to \H^0\DB_{X/F}^1$ is surjective, then it is locally $K_1$-regular if $m\ge \frac{d-1}{2}$;
        \item Let $s=\dim X_{\sing}$. Then $X$ is $K_{-t}$-regular if $t\ge \max\{s+1, d-2m-2+s\}$; if $X$ is in addition seminormal, then the bound can be improved to $t\ge \max\{s, d-2m-2+s\}$.
        \item If $X$ is either affine, or has isolated singularities, then we can remove the adjective ``locally'' in (1), (2) and (3).
    \end{enumerate}
\begin{proof}
    (1) By Theorem \ref{thm:CHWW-local}, we need to show
    \[\H^q C^p = 0 \text{ for } q-p\ge \max\{1, d-2m-2\},\]
    where $C^p = \Cone(\LL^p_{X/\QQ}\to \DB^p_{X/\QQ})$. Since $q\ge q-p\ge 1$, this is equivalent to 
    \[\H^q \DB_{X/\QQ}^p = 0 \text{ for } q-p\ge d-2m-2.\]
    Since $X$ has pre-$m$-Du Bois singularities, this holds when $p\le m$. When $p\ge m+1$, we have $q\ge d-2m-2+p\ge d-m-1$, in which case the vanishing follows from Proposition \ref{prop:DB-vanishing-over-subfield}.

    (2) By (1), it suffices to show $\H^0 C^0 \simeq \coker(\O_X\to \H^0 \DB_{X/\QQ}^0) =0$. This is true by \cite{Saito-seminormalization}*{Corollary 0.3}, since $X$ is assumed to be seminormal.

    (3) By (1) and (2), since $\H^{-1} C^0=0$, it suffices to show $\H^0 C^1 \simeq \coker(\Omega_{X/\QQ}^1\to \H^0 \DB_{X/\QQ}^1) =0$. This follows by assumption and Lemma \ref{lemma: DB-change-of-field}(1).

        (4) Consider the local-to-global spectral sequence
    \[
        E_2^{pq}=H^p(X,\NK_q)\implies NK_{q-p}(X).
    \]
    Since $\NK_q$ is supported on $X_{\sing}$ and $\dim X_{\sing}=s$, we have
    \[
        H^p(X,\NK_q)=0 \qquad \text{for } p>s.
    \]
    On the other hand, if $p-q\ge d-2m-2+s$, then for every $0\le p\le s$ we have
    \[
        q=(q-p)+p\le -(d-2m-2+s)+s = -d+2m+2.
    \]
    Note also that $q=p-t\le p-s-1<0$. Hence $\NK_q=0$ by part (1). It follows that $NK_{-t}(X)=0$ for $t\ge 1$ and $t\ge d-2m-2+s$. That is, $X$ is $K_{-t}$-regular. When $X$ is in addition seminormal, the assertion follows by invoking (2) in place of (1).
    
    (5) If $X$ is either affine, or has isolated singularities, the local-to-global spectral sequence 
\[E_2^{pq} = H^p(X,\NK_q(X))\implies NK_{q-p}(X)\]
degenerates, so local $K_{-d+2m+2}$-regularity implies $K_{-d+2m+2}$-regularity. 
\end{proof}
\end{proposition}

As an application of this Proposition, we give some examples that are $K_0$-regular:

\begin{example}
    We list some varieties with pre-$k$-Du Bois singularities for all $k$. These have ``mild" singularities from the perspective of Hodge theory. By Proposition \ref{prop:DB-implies-neg-regular-affine-case}, such an $X$ is locally
    \begin{enumerate}
        \item $K_n$-regular for all $n<0$;
        \item $K_0$-regular if $X$ is seminormal.
    \end{enumerate}
    
    Examples include: toric varieties \cite{GNPP}, finite or more generally geometric quotients (\cite{DB}), cones over varieties satisfying Bott vanishing \cite{K-cones}, \cite{SVV}*{Corollary 7.6}. Since all these varieties are normal, they are locally $K_0$-regular.
\end{example}

Next, we investigate the converse to the implication in Proposition \ref{prop:DB-implies-neg-regular-affine-case}:

\begin{proposition}[$\implies$ Theorem \ref{intro-thm:relation-lci}(3)]
    \label{prop:negative-regular-implies-DB}
    Let $X$ be a $d$-dimensional affine variety with $\dim X_{\sing} = s$. Suppose either of the following holds:
    \begin{enumerate}
        \item $X$ is a local complete intersection;
        \item $\depth \H^0\DB_X^p\ge d-p$ for all $p\ge 0$.
    \end{enumerate} 
    
    If $X$ is $K_{-d+2m+1+s}$-regular, then it is pre-$m$-Du Bois.
    \begin{proof}
        Since $X$ is $K_{-d+2m+1+s}$-regular, we know by Theorem \ref{thm:CHWW-main} that
        \[\H^i \DB_{X/\QQ}^j = 0 \text{ for $i\ge d-2m-1-s+j$ and $i>0$}.\]
        Applying Theorem \ref{thm:MP-vanishing} in case (1), and Theorem \ref{thm:PSV-vanishing} in case (2) (with $k=\QQ$ and $F$ the field of definition of $X$), we have
        \[\H^i\DB_{X/\QQ}^j = 0 \text{ for $1\le i\le d-s-2-j$}.\]
        For $j\le m$, the two inequalities of $i$ cover all indices $i\ge 1$. It follows that $X$ is pre-$m$-Du Bois.
        \end{proof}
    \end{proposition}    

Combining the two Propositions above, we obtain:
\begin{corollary}
    \label{cor:equiv-using-vanishing-thm}
    Let $X$ be a $d$-dimensional variety with isolated singularities. Suppose one of the following holds:
    \begin{enumerate}
        \item $X$ is a local complete intersection;
        \item $\depth \H^0\DB_X^p\ge d-p$ for all $p\ge 0$.
    \end{enumerate}
    Then
    \begin{align*}
        \text{$X$ has pre-$m$-Du Bois singularities} 
        &\iff \text{$X$ is $K_{-d+2m+1}$-regular}  \\
        &\iff \text{$X$ is $K_{-d+2m+2}$-regular}.    
    \end{align*}
    when $-d+2m+2<0$.
\begin{proof}
    By Proposition \ref{prop:DB-implies-neg-regular-affine-case} and Proposition \ref{prop:negative-regular-implies-DB} (applied with $s=0$), we have the implications
    \[\text{$K_{-d+2m+1}$-regular} \implies \text{pre-$m$-Du Bois} \implies \text{$K_{-d+2m+2}$-regular}.\]
    Since $K_{-d+2m+2}$-regular implies $K_{-d+2m+1}$-regular, all three conditions are in fact equivalent.
\end{proof}
\end{corollary}

When $m=0$, this gives:

\begin{corollary}\label{cor:equiv-using-vanishing-thm-k=0}
    Let $X$ be an affine variety of dimension $d\ge 2$. Suppose $X$ is normal, Cohen-Macaulay with isolated singularities. Then the following are equivalent:
    \begin{enumerate}
        \item $X$ has Du Bois singularities;
        \item $X$ has pre-$0$-Du Bois singularities;
        \item $X$ is $K_{-d+1}$-regular;
        \item $X$ is $K_{-d+2}$-regular.
    \end{enumerate}
    \begin{proof}
    Since $X$ is Cohen-Macaulay, condition (2) in Corollary \ref{cor:equiv-using-vanishing-thm} is satisfied with $m=0$. We thus have $(2)\iff (3) \iff (4)$.

    Since $X$ is normal, hence seminormal, the equivalence $(1)\iff (2)$ follows from the well-known fact that $\H^0\DB_{X/\QQ}^0 \simeq \H^0\DB_{X/\CC}^0 \simeq \O_{X^{sn}}$, the structure sheaf of the semi-normalization of $X$.
    \end{proof}
\end{corollary}

%\begin{remark}
%    The two results above fails when $d=1$. Indeed, a curve is pre-$m$-Du Bois for all $m\ge 0$, and $K_0$-regular if and only if it is seminormal.
%\end{remark}

\begin{remark}
    The assumptions in Corollary \ref{cor:equiv-using-vanishing-thm-k=0} are satisfied for normal surfaces, so we have the equivalences
    \begin{enumerate}
        \item $X$ is $K_{-1}$-regular;
        \item $X$ is $K_0$-regular;
        \item $X$ has Du Bois singularities;
        \item $X$ has pre-$0$-Du Bois singularities;
        \item $X$ has pre-$1$-Du Bois singularities.
    \end{enumerate}
    The equivalence of (1)-(3) with (4) is remarked in \cite{PSV}*{Example 3.5}. See also Lemma \ref{lemma:DB-table-constraint}.

    In \cite{Weibel-surface}*{Theorem 5.9}, characterizations for $K_{-1}$-regularity that are more geometric in nature are obtained. For example, given a resolution of singularities $\tilde{X}\to X$ with exceptional divisor $E$, it is shown in \textit{loc. cit.} that $X$ is $K_{-1}$-regular if and only if    
    \[H^1(E, \I/\I^n)=0 \text{ for all }n\ge 2,\]
    where $\I$ is the ideal of $\O_{\tilde{X}}$ defining $E$.

    One can directly see that this condition characterizes Du Bois singularities for normal surfaces, as follows: since $X$ is seminormal, we have $\H^0\DB_X^0=\O_X$; moreover, by dimension reasons $\H^2\DB_X^0=0$. It follows that $X$ has Du Bois singularities if and only if $\H^1 \DB_X^0=0$.
    
    By Steenbrink's triangle for the Du Bois complexes, 
\[\DB_X^0 \to \R f_*\O_{\tilde{X}}(-E)\to \O_x \xto{+1},\]
where $x$ is the isolated singularity of $X$. We thus see that
\[\H^1\DB_X^0 = R^1 f_*\O_{\tilde{X}}(-E),\]
which by the Formal Function Theorem is
\[R^1f_*\O_{\tilde{X}}(-E) = \varprojlim_n H^1(E_n, \O_{\tilde{X}}(-E)|_{E_n}) = \varprojlim_n H^1(E, I/I^n).\]
Since $H^2(E,\blank)=0$, from the short exact sequence
\[0\to I^n/I^{n+1} \to I/I^{n+1} \to I/I^{n}\to 0\]
we see that for all $n\ge 1$,
\[H^1(E,I/I^n)=0 \implies H^1(E,I/I^{n+1})=0.\]
It follows that
\begin{align*}
X \text{ is Du Bois} 
&\iff R^1 f_*\O_{\tilde{X}}(-E) =\varprojlim_n H^1(E, I/I^n)= 0 \\
&\iff  H^1(E, I/I^n) = 0 \text{ for sufficiently large $n$} \\
&\iff H^1(E, I/I^n) = 0 \text{ for all $n$}.
\end{align*}
I am grateful to Sung Gi Park for telling me about this argument. 
\end{remark}

\begin{remark}
    For a curve $C$, the works of Traverso \cite{Traverso}, Vorst \cite{Vorst} and Weibel \cite{Weibel-surface} show that 
    \begin{itemize}
        \item $C$ is always $K_{-1}$-regular;
        \item $C$ is $K_0$-regular $\iff$ $C$ is $K_1$-regular $\iff$ $C$ is $\Pic$-regular $\iff$ $C$ is seminormal. This fits into our Hodge theory picture, as a curve is seminormal if and only if it has Du Bois singularities.
        \item If $C$ is $K_2$-regular, then it is smooth.
    \end{itemize}
\end{remark}

With a bit more work, one can extend Corollary \ref{cor:equiv-using-vanishing-thm} beyond the bound $m\le \frac{d-3}{2}$ for local complete intersections of dimension $\ge 2$. We will need the smoothness criterion given by Theorem \ref{thm:MP-codim-bound}, as well as the following result about the torsion and cotorsion of the K\"ahler differentials. See \cite{Kahler-diff-hypersurface}*{Theorem 1.11} for the case of hypersurfaces, and \cite{Kahler-diff-lci}*{Corollary 3.1} for local complete intersections.
\begin{theorem}[Graf, Miller-Vassiliadou]\label{thm:reflexive-Kahler-diff}
    Let $X$ be a local complete intersection over $F$. Then the K\"ahler differential $\Omega_{X/F}^p$ is 
    \begin{enumerate}
        \item Torsion-free if and only if $p\le \codim X_{\sing}-1$;
        \item Reflexive, if and only if $p\le \codim X_{\sing}-2$. If this is the case, we have $\Omega_{X/F}^p\simeq \H^0\DB_{X/F}^p$.
    \end{enumerate}
\end{theorem}

\begin{corollary}\label{cor:equiv-lci-isolated}
    Let $X$ be a local complete intersection of dimension $d$ with isolated singularities. Then
    \begin{align*}
        \text{$X$ has $m$-Du Bois singularities} &\iff \text{$X$ is $K_{-d+2m+1}$-regular} \\
        &\iff \text{$X$ is $K_{-d+2m+2}$-regular}.
    \end{align*}
    \begin{proof}
        First, we show that
        \[\text{$m$-Du Bois} \implies \text{$K_{-d+2m+2}$-regular}.\]
        We check this case by case:
        \begin{itemize}
            \item If $-d+2m+2<0$, this follows by Proposition \ref{prop:DB-implies-neg-regular-affine-case}(1);
            \item If $-d+2m+2=0$, this follows by Proposition \ref{prop:DB-implies-neg-regular-affine-case}(2), because local complete intersections of dimension $\ge 2$ with isolated singularities are normal;
            \item If $-d+2m+2=1$, this follows by Proposition \ref{prop:DB-implies-neg-regular-affine-case}(3), because the surjectivity of $\Omega_{X/F}^1\to \H^0\DB_{X/F}^1$ (where $F$ is the field of definition of $X$) holds by Theorem \ref{thm:reflexive-Kahler-diff}, as long as $d\ge 3$;
            \item If $-d+2m+2\ge 2$, then $d\le 2m$. Since $X$ is $m$-Du Bois, it is smooth by Theorem \ref{thm:MP-codim-bound}, hence $K_{-d+2m+2}$-regular.
        \end{itemize}
    Next, we show that 
    \[ \text{$K_{-d+2m+1}$-regular}\implies\text{$m$-Du Bois}.\]
    By Proposition \ref{prop:negative-regular-implies-DB} applied with $s=0$, we know $X$ is pre-$m$-Du Bois. If $m\le d-2$, then $X$ is $m$-Du Bois by Theorem \ref{thm:reflexive-Kahler-diff}. Thus we may assume $m\ge d-1$. We consider several cases:
    \begin{itemize}
        \item If $d\ge 4$, then one can check that
        \[m\ge d-1 \ge \floor{\frac{d+1}{2}}\text{ and } \floor{\frac{d+1}{2}}\le d-2,\]
        so $X$ is $\floor{\frac{d+1}{2}}$-Du Bois, hence smooth by Theorem \ref{thm:MP-codim-bound};
        \item If $d=3$, then $m\ge 2$. By assumption $X$ is $K_2$-regular, so
        \begin{itemize}
            \item The map $\Omega_{X/\QQ}^2\to \H^0\DB_{X/\QQ}^2$ is surjective by Theorem \ref{thm:CHWW-local}. Since $X$ is $1$-Du Bois, Lemma \ref{lemma: DB-change-of-field}(1) implies the map $\Omega_{X/F}^2\to \H^0\DB_{X/F}^2$ is surjective;
            \item The map $\Omega_{X/F}^2\to \H^0\DB_{X/F}^2$ is injective by Theorem \ref{thm:reflexive-Kahler-diff}.
        \end{itemize}
        It follows that $X$ is $2$-Du Bois, hence smooth by Theorem \ref{thm:MP-codim-bound}.
        \item If $d=2$, then $m\ge 1$. By assumption $X$ is $K_1$-regular. Similar to the case $d=3$, one checks that $X$ is $1$-Du Bois, hence smooth.
        \item If $d=1$, then $m\ge 0$. If $m=0$, then $X$ is $K_0$-regular, hence $\Pic$-regular because $K_0(X)=\mathbb{Z}\oplus \Pic(X)$ for curves. By Traverso's theorem \cite{Traverso}*{Theorem 3.5}, $X$ is seminormal, hence has Du Bois singularities; when $m\ge 1$, $X$ is $K_2$-regular, hence smooth by Vorst's Theorem \cite{Vorst}*{Theorem 3.4}.
    \end{itemize}
\end{proof}
\end{corollary}

\iffalse
\begin{remark}    
    Corollary \ref{cor:equiv-lci-isolated} does not hold when $d=2$, because there exists pre-$1$-Du Bois ($\iff$ pre-$0$-Du Bois) surfaces that is not smooth ($\iff K_1$-regular). For example, $(x^2+y^2+z^2=0)\subset \AA^3_{\CC}$ is pre-$1$-Du Bois, but not $K_1$-regular.

    When $d=3$, the hypersurface $(x^2+y^2+z^2+w^2=0)\subset \AA^4_{\CC}$ has pre-$2$-Du Bois singularities, but is not $2$-Du Bois.
\end{remark}
\fi

\begin{corollary}[$\implies$ Theorem \ref{intro-thm:char-of-K-regularity}]
    \label{cor:K-regularity-min-exponent-lci}
    Let $X$ be a complex local complete intersection with isolated singularities of dimension $d$. Let $r$ be the codimension of $X$ in its embedding, and $\tilde{\alpha}_X$ be the minimal exponent of $X$ (which depends on the embedding). Then
    \[X \text{ is $K_m$-regular} \iff \tilde{\alpha}_X 
    \ge \ceil{\frac{m+d}{2}}+r-1.\]
    \begin{proof}
        Combining Corollary \ref{cor:equiv-lci-isolated} with the main theorems of \cites{JKSY, MOPW, CDM}, which states that
        \[\text{$X$ has $m$-Du Bois singularities} \iff \tilde{\alpha}_X\ge m+r,\]
        we obtain the result.
    \end{proof}
\end{corollary}

\begin{remark}\label{rmk:equiv-higher-dim-sing-locus}
    Similar arguments as in Corollary \ref{cor:equiv-lci-isolated} show that for an affine local complete intersection $X$ with $\dim X_{\sing}=1$, one has
    \[\begin{tikzcd}
        &K_{-d+1}\text{-regular}\ar[d, Leftarrow] & \\
        &K_{-d+2}\text{-regular} \ar[r,Leftrightarrow]\ar[d, Leftarrow] &\text{Du Bois}\\
        &K_{-d+3}\text{-regular}\ar[d, Leftarrow] & \\
        &K_{-d+4}\text{-regular}\ar[d, Leftarrow] \ar[r,Leftrightarrow] &1\text{-Du Bois} \\
        &\cdots
    \end{tikzcd}\]
    In particular, 
    \[\tilde{\alpha}_X\ge m+r \iff X \text{ is } K_{-d+2m+2}\text{-regular}.\]
    More generally, suppose $X$ is an affine local complete intersection $\dim X_{\sing}=s$. Then 
    \begin{align*}
        \text{$X$ is $K_{-d+2m+1+s}$-regular}\implies 
        &\text{$X$ has $m$-Du Bois singularities} \\
        \implies &\text{$X$ is $K_{-d+2m+2}$-regular}.
    \end{align*}
    Thus $K_m$-regularity can be read off, at least partially, from the minimal exponent $\tilde{\alpha}_X$ and the codimension $r$ of $X$ in its embedding:
    \begin{enumerate}
        \item $X$ is $K_m$-regular for $m\le -d+2\tilde{\alpha}_X-2r+2$;
        \item $X$ is not $K_m$-regular for $m=-d+2\ceil{\tilde{\alpha}_X}-2r+s$.
    \end{enumerate}
\end{remark}

\section{Application to regularity criteria}\label{sec:regularity}

In this section, we prove the following regularity criterion. Recall that $K_m$-regularity implies $K_{m-1}$-regularity for all $m$ \cites{Vorst,Dayton-Weibel}, hence it immediately implies Theorem \ref{intro-thm:regularity}. 

\begin{theorem}($\implies$ Theorem \ref{intro-thm:regularity})\label{thm:regularity}
Let $X$ be an affine local complete intersection, with singular locus $X_{\sing}$. Then
    \begin{enumerate}
        \item If $\codim X_{\sing}$ is odd, and $X$ is $K_2$-regular, then it is regular;
        \item If $\codim X_{\sing}$ is even, and $X$ is $K_1$-regular, then it is regular.
    \end{enumerate}
\end{theorem}

 The proof is similar to that of Corollary \ref{cor:equiv-lci-isolated}, and relies on three main ingredients:
\begin{enumerate}
    \item A regularity criterion due to \Mustata-Popa (Theorem \ref{thm:MP-codim-bound}). The relevance of this result in relation to extensions of Vorst's conjecture is observed in \cite{K-DB};
    \item Results due to Graf and Miller-Vassiliadou about the torsion and cotorsion of the K\"ahler differentials of local complete intersections (Theorem \ref{thm:reflexive-Kahler-diff});
    \item The relationship between $K_m$-regularity and pre-$k$-Du Bois singularities (Proposition \ref{prop:negative-regular-implies-DB}), which in turn relies on a vanishing theorem of \Mustata-Popa (Theorem \ref{thm:MP-vanishing}).
\end{enumerate}

\begin{proof}[Proof of Theorem \ref{intro-thm:regularity}]
    Let $F$ be the field of definition of $X$. By Theorem \ref{thm:MP-codim-bound}, regularity of $X$ is implied by having $m$-Du Bois singularities for $m> \frac{d-1}{2}$.

(1) Suppose $\codim X_{\sing}=2m+1$ is odd. To show $X$ is smooth, it suffices to show $X$ has $(m+1)$-Du Bois singularities, i.e. the following two conditions hold:
\begin{itemize}
    \item $X$ has pre-$(m+1)$-Du Bois singularities. By Proposition \ref{prop:negative-regular-implies-DB}, this is implied by $K_{-d+2(m+1)+1+s} = K_2$-regularity;
    \item The maps $\Omega_{X/F}^p\to \H^0\DB_{X/F}^p$ are isomorphisms for $p\le m+1$. By Theorem \ref{thm:reflexive-Kahler-diff}, this is the case when $p\le \codim X_{\sing} - 2$. Thus we have the desired isomorphism when $m+1\le \codim X_{\sing}-2 = (2m+1)-2$, i.e. when $m\ge 2$.
    
    When $m=1$, we need to show $\Omega_{X/F}^2\to \H^0\DB_{X/F}^2$ is an isomorphism. By Theorem \ref{thm:CHWW-main}, $K_2$-regularity implies that this map is surjective with $F$ replaced by $\QQ$, and Lemma \ref{lemma: DB-change-of-field}(2) implies the surjectivity over $F$; on the other hand Theorem \ref{thm:reflexive-Kahler-diff} implies that $\Omega_{X/F}^2$ is torsion-free, so the map is an isomorphism.

    Finally, it is impossible that $m=0$, as $K_2$-regularity implies $X$ has $1$-Du Bois singularities, so $\codim X_{\sing} \ge 3$.
\end{itemize}
Thus $X$ has $(m+1)$-Du Bois singularities, hence smooth.

(2) Suppose $\codim X_{\sing}=2m$ is even. To show $X$ is smooth, it suffices to show $X$ has $m$-Du Bois singularities, i.e. the following two conditions hold:
\begin{itemize}
    \item $X$ has pre-$m$-Du Bois singularities. By Proposition \ref{prop:negative-regular-implies-DB}, this is implied by $K_{-d+2m+1+s} = K_1$-regularity;
    \item The maps $\Omega_{X/F}^p\to \H^0\DB_{X/F}^p$ are isomorphisms for $p\le m$. By Theorem \ref{thm:reflexive-Kahler-diff}, this is the case when $p\le \codim X_{\sing} - 2$. Thus we have the desired isomorphism when $m\le \codim X_{\sing}-2 = 2m-2$, i.e. when $m\ge 2$.
    
    When $m=1$, we need to show $\Omega_{X/F}^1\to \H^0\DB_{X/F}^1$ is an isomorphism. By Theorem \ref{thm:CHWW-main} and Lemma \ref{lemma: DB-change-of-field}(1), $K_1$-regularity implies that this map is surjective; on the other hand the results \cites{Kahler-diff-hypersurface,Kahler-diff-lci} show that $\Omega_X^1$ is torsion-free, so the map is injective.
\end{itemize}
Thus $X$ has $m$-Du Bois singularities, hence smooth.
\end{proof}

\begin{remark}
    The proof above can be used to show the following fact: if $X$ is a local complete intersection that is pre-$m$-Du Bois for all $m$, then either $\codim X_{\sing}\le 3$ or $X$ is smooth.
\end{remark}

The result above can be improved for local complete intersections of minimal embedding codimension $r\ge 2$:

\begin{definition}\label{def:minimal-embedding-codim}
    The \textit{minimal embedding codimension} of a local complete intersection $X$ is the minimal integer $r$ such that for every $x\in X$, there exists a regular local ring $R$, and a regular sequence $(f_1, \cdots, f_r)$ in $R$, such that
    \[\O_{X,x}\simeq R/(f_1, \cdots, f_r).\]
\end{definition}

The following theorem, which I learned from B. Dirks, shows that the codimension bound in Theorem \ref{thm:MP-codim-bound} can be improved when $X$ is a local complete intersection of minimal embedding codimension $r$. The proof follows as in \cite{CDM}*{Corollary 1.3}. See \cite{K-cones}*{Theorem 3.5} for a detailed argument. 

\begin{theorem}[Bradley Dirks]\label{thm:codim-bound}
    Let $X$ be a local complete intersection of minimal embedding codimension $r$. Suppose $X$ has $m$-Du Bois singularities. Then
    \[\codim X_{\sing}\ge 2m+r.\]
\end{theorem}

As a consequence, we have

\begin{corollary}\label{cor:regularity-min-dim-r}
    Let $X$ be an affine local complete intersection of minimal embedding codimension $r$. Then
    \begin{enumerate}
        \item If $\codim X_{\sing} - r\ge 6-2r$ is even, and $X$ is $K_{3-r}$-regular, then $X$ is smooth.
        \item If $\codim X_{\sing} - r\ge 5-2r$ is odd, and $X$ is $K_{2-r}$-regular, then $X$ is smooth.
    \end{enumerate}
\begin{proof}
    By Theorem \ref{thm:codim-bound}, a variety $X$ with $m$-Du Bois singularities satisfies $\codim X_{\sing} \ge 2m+r$. Thus,

    (1) When $\codim X_{\sing} = 2m+r$, to show smoothness we need to show $X$ is $(m+1)$-Du Bois, which is the case if
    \begin{itemize}
        \item $X$ is pre-$(m+1)$-Du Bois. This is implied by $K_{-d+2(m+1)+1+s}=K_{3-r}$-regularity.
        \item $\Omega_X^p$ is reflexive for $p\le m+1$. This is the case if $m+1\le \codim X_{\sing} - 2 = 2m+r-2$, i.e. when $m\ge 3-r$.
    \end{itemize}
    The proof of (2) is similar.
\end{proof}
\end{corollary}

\begin{remark}
    It is desirable to know about the implications of $K_2$-regularity beyond affine local complete intersections. For local complete intersections with isolated singularities (that are not necessarily affine), Theorem \ref{intro-thm:regularity} and Corollary \ref{cor:regularity-min-dim-r} continue to hold, thanks to the degeneration of the local-to-global spectral sequence
    \[E_2^{pq} = H^p(X, \N\K_q(X)) \implies NK_{q-p}(X).\]
    When the singular locus has higher dimension, it is not known whether $K_2$-regular implies smooth; see Question \ref{ques:K2-reg-non-affine}.
    \iffalse
    When the singular locus has higher dimension, this is expected to fail.

    For non-local complete intersections, it is shown in \cite{K-DB} that $K_2$-regularity implies normal.                                 
    \fi
\end{remark}

\begin{remark}\label{rmk:K2-for-non-lci}
    By Theorem \ref{thm:CHWW-main}, if an affine variety $X$ is $K_2$-regular, then the natural map
    \[\Omega_X^1\to \H^0\DB_X^1\]
    is an isomorphism. In the study of higher Du Bois singularities beyond the local complete intersection case, one difficulty is that there is no known variety satisfying this condition; see Remark \ref{rmk:no-non-lci-example}. This suggests also that $K_2$-regular varieties that are not local complete intersections (if there are any) are tricky to find.
\end{remark}

\begin{remark}\label{rmk:K1-regular-surface}
    By Theorem \ref{intro-thm:regularity}, a normal $K_1$-regular surface is smooth if it is a local complete intersection. It is natural to ask if there is any non-lci surface that is $K_1$-regular but not smooth. By results in \cite{Srinivas} and \cite{K-cones}, one can show that no affine cone over a smooth curve can be $K_1$-regular. I thank V. Srinivas for telling me about these references.

    Note that the normality assumption is necessary: the surface $(xyz=0)\subset \AA^3$ is $K_1$-regular. In fact, any affine simple normal crossings divisor is $K_1$-regular.
\end{remark}

\section{Bass' question and the Du Bois invariants}\label{sec:DB-table}
Let $X$ be a variety over $F$. The Du Bois invariants, defined as
\[b^{p,q} = \dim_F \H^q \DB_{X/F}^p \]
for $q>0$, are introduced by Steenbrink \cite{Steenbrink-isolated} when $X$ has isolated singularities. It follows from Proposition \ref{prop:DB-complex-support} that $b^{p,q}=0$ for all $q\ge d:=\dim X$. We arrange the invariants $b^{p,q}$ for $0\le p \le d$ and $1\le q\le d-1$ into a $(d-1)\times (d+1)$ table, and refer to it as the \textit{Du Bois table}.

For surfaces over $\QQ$, we have a $1\times 3$ table of the form:
\[
    \begin{tabular}{|c c c}
        $b^{0, 1}$ & $b^{1,1}$ & $b^{2,1}$ \\
        \hline
    \end{tabular}
\]

In fact, it is well known that for surfaces, the only interesting invariants are $b^{0,1}$ and $b^{1,1}$:
\begin{lemma}[\cite{K-Bass-negative}*{Proposition 3.2}]
    Let $X$ be a surface over $F$. For $q>0$ and $p\ge 0$, we have 
    \[\H^q\DB_{X/F}^p = 0\]
    unless $(p,q)=(0,1)$ or $(1,1)$.
    \begin{proof}
        By Steenbrink's vanishing theorem (see Theorem \ref{thm:Steenbrink-vanishing}), we have $\H^q\DB_{X/F}^p = 0$ when $p+q>2$. The vanishing when $(p,q)=(0,2)$ follows from Proposition \ref{prop:DB-complex-support}(1).
    \end{proof}
\end{lemma}

Applications of the Du Bois invariants to $K$-theory started in \cite{K-Bass-negative}, where the authors used an example of surface with $b^{0,1}=1$ and $b^{1,1}=0$ to give negative answers to Bass' question \cite{Bass}, which asks whether
\[K_m(X)\simeq K_m(X\times \AA^1) \implies K_m(X)\simeq K_m(X\times \AA^2).\]

Indeed, for surfaces, the main theorem of \cite{K-Bass} implies that Bass' question for the negative $K$-groups has an affirmative answer if and only if $b^{0,1}$ and $b^{1,1}$ are both zero, or both nonzero. In view of the potential applications of these invariants to $K$-theory, it is desirable to have a better understanding of them. In this section, we give a complete classification of these invariants for surfaces.

First, along the lines of \cite{PSV}*{Proposition C}, we show that $b^{0,1}$ and $b^{1,1}$ are constrained by the inequality $b^{0,1}\ge b^{1,1}$:
\begin{lemma}
    \label{lemma:DB-table-constraint}
    Let $X$ be a surface over $F$. Then the map
    \[d: \H^1\DB_{X/F}^0\to \H^1\DB_{X/F}^1\]
    is surjective. In particular, we have $b^{0,1}\ge b^{1,1}$ when $X$ has isolated singularities.
\begin{proof}
    By Lemma \ref{lemma: DB-change-of-field}(2), it suffices to prove the result when $F=\CC$. In this case, we have by the Poincar\'e Lemma that
    \[\CC\xto{\sim} (\DB_{X/\CC}^{\bullet})^{\an}\]
    as sheaves in the analytic topology on $X(\CC)$.

    By construction of the Du Bois complex (in the analytic category, see \cite{DB}), we have a spectral sequence
    \[E_1^{pq} = \H^q(\DB_{X/\CC}^p)^{\an} \mapsto \H^{p+q}(\DB_{X/\CC}^{\bullet})^{\an}\]
    converging to $(\DB_{X/\CC}^{\bullet})^{\an}\simeq \CC[0]$ supported in cohomological degree zero. Analyzing the maps in the spectral sequence shows that there is a surjection
    \[d^{\an}: \H^1(\DB_{X/\CC}^0)^{\an}\to \H^1(\DB_{X/\CC}^1)^{\an}.\]
    Using Serre's GAGA theorem, we then conclude that the map $d$ is surjective.
\end{proof}
\end{lemma}

In fact, the invariants $b^{0,1}$ and $b^{1,1}$ can take on arbitrary values:
\begin{proposition}
    \label{prop:DB-invariants}
    Let $X$ be a normal surface over $\QQ$.
    \begin{enumerate}
        \item If $X$ is an affine cone over a projective curve, or a quasi-homogeneous hypersurface, then $b^{0,1}=b^{1,1}$; in these cases, Bass' question for $K_m$ has positive answer when $m\le 0$. 
        \item Let $m\ge n\ge 0$ be any pair of nonnegative integers. The hypersurface $(x^3+y^3+z^{3m}+xyz^{m+n}=0)$ has Du Bois invariants
        \[b^{0,1}= m \quad \text{ and }\quad b^{1,1}=n.\]
    \end{enumerate}
    \begin{proof}
        (1) When $X$ is an affine cone over a projective curve $C$ associated to a line bundle $L$ on $C$, explicit calculation as in \cite{K-cones} and \cite{PSh} shows that
        \[\H^1\DB_{X/\QQ}^0=\H^1\DB_{X/\QQ}^1\simeq \bigoplus_{m\ge 0} \HH^1(C,\DB_{C/\QQ}^0\otimes L^m).\]
        For quasi-homogeneous hypersurfaces, we use van Straten-Steenbrink's computation of the geometric genus $p_g\le b^{0,1}$ and the \textit{alpha invariant} $\alpha:=b^{0,1}-b^{1,1}$: by \cite{Steenbrink-V-filtration}*{Theorem 3.1, Corollary 6.3}, for a surface $X=(f=0)\subset \AA^3=Y$ with isolated singularities at $0$, we have
        \[p_g = \dim Q^f/V_{>0}\]
        and 
        \[p_g+\alpha = \dim Q^f/\overline{K}\cap V_{>0},\]
        where $Q^f = \Omega_{Y,0}^3/df\wedge \Omega_{Y,0}^2$ and $\overline{K}$ is the kernel of the multiplication by $f$ on $Q^f$. If $f$ is quasi-homogeneous, multiplication by $f$ is zero on $Q^f$, so $\overline{K}=Q^f$, it follows that $\alpha = b^{0,1}-b^{1,1}=0$, as claimed. 
        
        (2) For semi-quasi-homogeneous hypersurfaces of weight $(w_x, w_y, w_z)$ with isolated singularities, the $V$-filtration on $Q^f$ is computed in \cite{Saito-V-filtration}. One has that $V_{>l} Q^f$ is generated by monomials $x^a y^b z^c$ satisfying
        \[w_x (a+1)+w_y(b+1)+w_z(c+1) - 1 > l.\]
        In particular, for $f=x^3+y^3+z^{3m}+xyz^{m+n}$, we have
        \[V_{>0} = (x, y, z^m)\]
        and a Gr\"obner basis computation shows that
        \[\overline{K} = (x, y, z^{2m-n})\]
        This gives that
        \[p_g = \dim Q^f/V_{>0} = m \ge b^{0,1}\]
        and 
        \[p_g+\alpha = \dim Q^f/\overline{K}\cap V_{>0} = 2m-n.\]
        Next, using the deformation argument in \cite{K-Bass-negative} (see Theorem 2.14 and the proof of Proposition 4.3 in \textit{loc. cit.}), we see that $b^{0,1}$ is independent of $n$. When $n=0$, we have $\alpha=m$, and since $0\le b^{1,1}=b^{0,1}-\alpha=b^{0,1}-m\le 0$, we have $b^{0,1}=m$ and $b^{1,1}=0$. It follows that for $n\ge 1$, we have
        \[b^{0,1}= m\quad \text{ and } \quad b^{1,1} = n,\]
        as required.
    \end{proof}
\end{proposition}

Moving to higher dimensions, we show that for homogeneous hypersurfaces with isolated singularities, Bass' question has positive answer for the non-positive $K$-groups. The proof relies on the computation of the cohomology groups $H^q(S, \Omega_S^p(l))=0$ for smooth hypersurfaces in projective spaces. Complete computations can be found in \cite{Duc}. We state part of the theorem below in the form of what we need:
\begin{theorem}\label{thm:Bott-vanishing}
    Let $S$ be a smooth hypersurface in $\PP^n$ of degree $d$. Then 
    \begin{enumerate}
        \item For $l\ge 0$, we have $H^q(S, \Omega_S^p(l)) = 0$ if $p+q\neq \dim S$;
        \item For $p+q=\dim S$,
        \[H^q(S, \Omega_S^p(l)) = 0 \text{ for all $l\ge 1$ }\iff (p+1)d \le n+1.\]
    \end{enumerate}
    \begin{proof}
        The result follows from \cite{Duc}*{Theorem D, Proposition 1, Remark 1}, and also the paragraphs before the proof of Theorem D, when $d\ge 2$ and $n\ge 4$.

        When $d=1$, the result follows from Bott vanishing for projective spaces; when $n=1$, the statement is vacuous; when $n=2,3$, one can compute $H^q(S, \Omega_S^p(l))$ using the long exact sequence associated to the following short exact sequences (see \cite{Duc}*{Proposition 1} for one implication in (2) when $n=3$):
        \[0\to \O_{\PP^n}(-d+l)\to \O_{\PP^n(l)}\to \O_S(l)\to 0, \quad n=2,3\]
        \[0\to \Omega_{\PP^3}^1|_S(l)\to \O_S(l-1)^{\oplus 4}\to \O_S(l)\to 0\]
        \[0\to \O_S(l-d)\to \Omega_{\PP^3}^1|_S(l)\to \Omega_S^1(l)\to 0\]
        and verify that the results hold true.
    \end{proof}
\end{theorem}

\begin{proposition}\label{prop:Bass-for-homog-hypsurf}
    Let $X$ be a homogeneous hypersurface with at worst isolated singularities. Then for every $m$,
    \[K_m(X)\xto{\sim} K_m(X\times \AA^1) \implies K_m(X)\xto{\sim} K_m(X\times \AA^r)\]
    for all $r\ge 1$. In particular, if $K_2(X)\simeq K_2(X\times \AA^1)$, then $X$ is smooth.

    \begin{proof}
        Since $X$ has isolated singularities, we may assume it is affine. Let $F$ be the field of definition of $X$. If $F$ has infinite transcendence degree over $\QQ$, then by \cite{K-Bass}*{Theorem 0.2(b)}, Bass' question has positive answer. Thus we may assume $F$ is algebraic over $\QQ$. 

        By \cite{K-Bass}*{Theorem 0.1}, the implication we need to show is equivalent to
        \[NK_m(X)=0\implies NK_i(X)=0 \text{ for $i<m$}. \]
        Applying \cite{K-Bass}*{Theorem 0.4}, this is in turn equivalent to  
        \[\H^qC^p=0 \text{ for $q-p=-m$} \implies \H^qC^p=0 \text{ for $q-p=-m-1$},\]
        where $C^p = \Cone(\LL_{X/\QQ}^p\to \DB_{X/\QQ}^p)$. Since $F$ is algebraic over $\QQ$, we have $\Omega_{F/\QQ}^p=0$ for all $p\ge 1$, so the spectral sequence in Lemma \ref{lemma: DB-change-of-field}(1) degenerates, giving that
        \[\H^q\DB_{X/F}^p \simeq \H^q\DB_{X/\QQ}^p\]
        and
        \[\H^q\LL_{X/F}^p \simeq \H^q\LL_{X/\QQ}^p.\]
        Note that since $X$ is a hypersurface with isolated singularities, as in the proof of Lemma \ref{lemma:lci-DB-def}, we have $\LL_{X/F}^p \simeq \Omega_{X/F}^p$ for $p\le \dim X$. Moreover, by Theorem \ref{thm:reflexive-Kahler-diff}, we have $\H^0 C^p=0$ for $p\le \dim X-2$ and $\H^{-1} C^p = 0$ for $p\le \dim X-1$. Consequently, the exact sequences in \cite{K-cones}*{Proposition 1.9} gives the isomorphisms
        \[\tors \Omega_{X/F}^{n} \simeq \tors \Omega_{X/F}^{n + 1}\]
        and
        \[\H^0\DB_{X/F}^{n-1}/\Omega_{X/F}^{n-1} \simeq \H^0\DB_{X/F}^{n}/\Omega_{X/F}^{n},\]
        where $n=\dim X$. Therefore, it suffices for us to check
        \[\H^q\DB_{X/F}^p=0 \text{ for $q-p=-m$} \implies \H^q\DB_{X/F}^p=0 \text{ for $q-p=-m-1$}.\]

        Since $X$ is a homogeneous hypersurface with isolated singularities, we can view it as an affine cone over a smooth projective hypersurface $S$. Using the computation of $\H^q\DB_{X/F}^p$ for cones \cite{K-cones}, as well as Theorem \ref{thm:Bott-vanishing}, we see that the Du Bois table of $X$ takes the following form:
        \begin{center}
        \adjustbox{scale=1}{
    \begin{tabular}{|c c c c}
        $b^{0, d} = 0$ & &\\
        $b^{0,d-1}=\underset{m\ge 1}{\bigoplus} H^{d-1}(\O_S(m))$ & $b^{1,d-1}=\underset{m\ge 1}{\bigoplus} H^{d-1}(\O_S(m))$ & \\
        & $b^{1,d-2}=\underset{m\ge 1}{\bigoplus} H^{d-2}(\Omega^1_S(m))$ & $b^{2,d-2}=\underset{m\ge 1}{\bigoplus} H^{d-2}(\Omega^1_S(m))$ \\
        & & $b^{2, d-3} =\underset{m\ge 1}{\bigoplus} H^{d-3}(\Omega_S^2(m))$ \\
        & & &$\cdots$ \\
        \hline
    \end{tabular}}\end{center}
        Thus, we only need to check
        \[H^q(\Omega^p_S(m))=0 \text{ for all $m\ge 1$} \implies H^{q+1}(\Omega^{p-1}_S(m))=0\text{ for all $m\ge 1$}\]
        for $p+q = \dim S$. This follows from Theorem \ref{thm:Bott-vanishing}(2).
    \end{proof}
\end{proposition}

\begin{remark}
    With notations as in the proof above, we have that for $p\le \dim X-2$, the following are equivalent:
    \begin{enumerate}
        \item $X$ has (pre-)$p$-Du Bois singularities;
        \item $X$ is $K_{-d+2+2p}$-regular;
        \item $NK_{-d+2+2p}(X)=0$;
        \item $\deg f\le \frac{d+1}{p+1}$.
    \end{enumerate}
    For $X$ to be singular, one needs $2\le \deg f\le \frac{d+1}{p+1}$, in which case $-d+2+2p\le 1$, so $X$ can be at most $K_1$-regular, as predicted by Theorem \ref{intro-thm:regularity}.
\end{remark}

\section{\texorpdfstring{$K$-regularity for projective varieties}{K-regularity for projective varieties}}
\label{sec:K-reg-for-proj-var}
In this section, we apply the main theorem of \cites{K-Weibel} to characterize $K$-regularity for projective varieties. I thank Dori Bejleri, Anh Duc Vo, and especially Elden Elmanto for discussions that shaped the contents presented here.

Our main result is the following characterization of $K_m$-regularity for projective varieties over a field $F$ of characteristic zero. 
\begin{theorem}[$\implies$ Theorem \ref{intro-thm:CHWW-projective}]
    \label{thm:CHWW-projective}
    Let $X$ be a projective variety over $F$. Then $X$ is $K_m$-regular if and only if
    \[\HH^i(X,\LL_{X/F}^p)\xto{\sim} \HH^i(X, \DB_{X/F}^p)\]
    are isomorphisms for $i-p\ge -m+1$.
\end{theorem}

\begin{remark}
    Let $X$ be a projective variety of dimension $d$. It follows from Theorem \ref{thm:CHWW-projective} that
    \[\text{$X$ is $K_{-d+1}$-regular} \iff \HH^d(X,\O_X)\to \HH^d(X,\DB_X^0)\text{ is an isomorphism}.\]
    In particular, if $X$ has Du Bois singularities then it is $K_{-d+1}$-regular, which is consistent with Proposition \ref{prop:relation-general}(1).
\end{remark}

By \cite{K-Weibel}*{Theorem 5.5}, there is a homotopy fibration sequence
\[\tilde{C_j}\HC\to \tilde{C}_jL_{cdh}\HC\to \tilde{C_j} \K,\]
where $\K$ denotes the non-connective $K$-theory spectrum of perfect complexes on $\rm Sch/F$, $\tilde{C}_j \E$ denotes the cofiber of the map $\E\to \E(-\times \AA^j)$ for a presheaf of spectra on $\rm Sch/F$, and $\HC$ is the presheaf of complexes computing cyclic cohomologies over $\QQ$. Following \textit{loc. cit.}, we write $H^i_{cdh}$ for the cohomology taken in cdh topology, and by $H^i$ the cohomology taken in Zariski topology. By definition, a scheme $X$ is $K_m$-regular if and only if $\pi_m\tilde{C_j}\K(X)=\ker(K_m(X)\to K_m(X\times \AA^j))=0$ for all $j\ge 0$.

Taking cohomology, we obtain a long exact sequence
\begin{align}\label{eqn:K-regularity-LES}
\begin{split}
\cdots \to &H^{d-1}(\tilde{C}_j\HC(X))\to H^{d-1}_{cdh}(\tilde{C}_j \HC(X))\to \pi_{-d+1}\tilde{C}_j \K(X) \\
\to & H^{d}(\tilde{C}_j\HC(X))\to H^{d}_{cdh}(\tilde{C}_j \HC(X))\to \pi_{-d}\tilde{C}_j \K(X) \to H^{d+1}(\tilde{C}_j\HC(X))
\end{split}
\end{align}

Thus, Theorem \ref{thm:CHWW-projective} can be proved by analyzing the surjectivity of the maps
\[H^i(\tilde{C}_j\HC(X))\to H^i_{cdh}(\tilde{C}_j \HC(X)).\]
The main Hodge-theoretic input for the proof is the following result, whose statement goes back to at least Du Bois \cite{DB} when $F=\CC$ and $p=0$:

\begin{proposition}\label{prop:surjectivity-over-C}
    Let $X$ be a projective variety over $F$. Then there is a surjection
    \[\HH^i(X, L\Omega_{X/F}^{\le p})\to \HH^i(X, \DB_{X/F}^{\le p})\]
    for all integers $i$ and $p\ge 0$.
    \begin{proof}
        By standard base change arguments (see the footnote to the proof of Lemma \ref{lemma: DB-change-of-field}), we reduce to the case when $F=\CC$.
        When $X$ is a complex projective variety, the Hodge-to-de-Rham spectral sequence
        \[E_1^{p,q}=H^q(X,\DB_{X/\CC}^p)\implies \HH^{p+q}(X,\CC)\]
        degenerates at $E_1$, giving a surjection
        \[H^i(X,\CC)\to H^i(X,\CC)/F^{p+1}H^i(X,\CC) = \HH^i(X,\DB_{X/\CC}^{\bullet}/F^{p+1}\DB_{X/\CC}^{\bullet})=:\HH^i(X,\DB_{X/\CC}^{\le p})\]
        for all integers $i$ and $p\ge 0$. 

        By \cite{Bhatt-derived}, the derived de Rham complex $\widehat{\rm dR}_{X/\CC}$ computes the Betti cohomology of $X(\CC)$, so we have the following commutative diagram
        \[\begin{tikzcd}
            &H^i(X,\CC) \ar[r,two heads] &H^i_{cdh}(X,\Omega_{X/\CC}^{\le p})\simeq \HH^i(X,\DB_{X/\CC}^{\le p}) \\
                &H^i(X,\widehat{\rm dR}_{X/\CC})\ar[u,no head, "\sim"] \ar[r] &H^i(X,L\Omega_{X/\CC}^{\le p}),\ar[u]
        \end{tikzcd}\]
        Since the top horizontal row is a surjection, we have that the maps
        \[H^i(X,L\Omega_{X/\CC}^{\le p})\to H^i_{cdh}(X,\Omega_{X/\CC}^{\le p})\simeq \HH^i(X,\DB_{X/\CC}^{\le p})\]
        are surjective, for all integers $i$ and $p\ge 0$.
    \end{proof}
\end{proposition}

We are now ready to prove the main theorem of this section.

\begin{proof}[Proof of Theorem \ref{thm:CHWW-projective}]
    First assume $X$ is $K_m$-regular. We want to show
    \[\HH^i(X,\LL_{X/F}^p)\xto{\sim} \HH^i(X, \DB_{X/F}^p)\]
    are isomorphisms when $i-p\ge -m+1$. By (\ref{eqn:K-regularity-LES}), we have that
    \[H^i(X, \tilde{C}_j\HC(X))\xto{\sim} H^i_{cdh}(X,\tilde{C}_j\HC(X))\]
    are isomorphisms for $i\ge -m+1$. Using the decomposition of cyclic homology
    \[\HC(X) = \bigoplus_{p\ge 0} L\Omega_{X/\QQ}^{\le p}[2p],\]
    and the fact that the cdh-sheafification of $L\Omega_{X/\QQ}^{\le p}$ is $\DB_{X/\QQ}^{\le p}$, we see that 
    \[\HH^{i+2p}(X, \tilde{C}_jL\Omega_{X/\QQ}^{\le p})\xto{\sim} \HH^{i+2p}_{cdh}(X,\tilde{C}_j\DB_{X/\QQ}^{\le p})\]
    are isomorphisms for all $i\ge -m+1$ and $p\ge 0$\footnote{It is unclear that this directly implies
    \[\HH^{i+2p}(X, L\Omega_{X/\QQ}^{\le p})\xto{\sim} \HH^{i+2p}_{cdh}(X,\DB_{X/\QQ}^{\le p})\]
    are isomorphisms for all $i\ge -m+1$ and $p\ge 0$, so the argument that follows takes a somewhat roundabout path.}. 
    \vskip \medskipamount
    \noindent\textit{Claim 1:} The maps $\HH^{i+2p}(X, \tilde{C}_jL\Omega_{X/\QQ}^{\le p})\xto{\sim} \HH^{i+2p}_{cdh}(X,\tilde{C}_j\DB_{X/\QQ}^{\le p})$ are isomorphisms for all $i\ge -m+1$ and $p\ge 0$. By induction, it suffices to show $\HH^{-m+1+2p}(X, \tilde{C}_jL\Omega_{X/\QQ}^{\le p})\xto{\sim} \HH^{-m+1+2p}_{cdh}(X,\tilde{C}_j\DB_{X/\QQ}^{\le p})$ are isomorphisms for all $p\ge 0$.
    
    Applying $\tilde{C}_j$ in Lemma \ref{lemma: DB-change-of-field}(1), we obtain spectral sequences
    \[E_1^{ij}=\Omega^i_{F/\QQ}\otimes \HH^j(X, L\Omega_{X/F}^{\le p-i})\implies \HH^{i+j}(X, L\Omega_{X/\QQ}^{\le p})\]
    \[E_1^{ij}=\Omega^i_{F/\QQ}\otimes \HH^j(X, \DB_{X/F}^{\le p-i})\implies \HH^{i+j}(X, \DB_{X/\QQ}^{\le p})\]

    Analyzing the spectral sequence, we see that it is enough to show the maps
    \[\HH^{-m+2p-1-i}(X,\tilde{C}_j L\Omega_{X/F}^{\le p-i})\to \HH^{-m+2p-1-i}(X,\tilde{C}_j \DB_{X/F}^{\le p-i})\]
    are surjective for all $i\ge 0$. This is implied by the surjectivity of the maps
    \[\HH^{-m+2p-1-i}(X,L\Omega_{X\times \AA^j/F}^{\le p-i})\to \HH^{-m+2p-1-i}(X,\DB_{X\times \AA^j/F}^{\le p-i})\]
    for $i\ge 0$. To prove the latter surjectivity, we analyze the spectral sequences, proved along the same lines as Lemma \ref{lemma: DB-change-of-field}:
    \[E_1^{ij}= \HH^j(p^*L\Omega^i_{X/F}\otimes \Omega_{X\times \AA^j/X}^{\le p-i})\implies \HH^{i+j}(L\Omega_{X\times \AA^j/k}^{\le p})\]
    \[E_1^{ij}= \HH^j(p^*\DB^i_{X/F}\otimes\Omega_{X\times \AA^j/X}^{\le p-i})\implies \HH^{i+j}(\DB_{X\times \AA^j/k}^{\le p})\]
    We need to check surjectivity of the maps $E_1^{ij}\to \tilde{E}_1^{ij}$ for $i+j=-m+2p-1-i$. That is, we need to show the maps
    \[\HH^{-m+2p-1-i-k}(p^*L\Omega_{X/F}^k\otimes \Omega_{X\times \AA^j/X}^{\le p-i-k})\to \HH^{-m+2p-1-i-k}(p^*\DB_{X/F}^k\otimes \Omega_{X\times \AA^j/X}^{\le p-i-k})\]
    are surjective for $0\le k\le p-i$. By induction using long exact sequences associated to
    \[\Omega_{X\times \AA^j/X}^{p-i-k}[-p+i+k]\to \Omega_{X\times \AA^j/X}^{\le p-i-k}\to \Omega_{X\times \AA^j/X}^{\le p-i-k-1}\xto{+1},\]
    it suffices to show the maps
    \[\HH^{-m+p-1}(L\Omega_{X/F}^k)\to \HH^{-m+p-1}(\DB_{X/F}^k)\]
    are surjective, for all $p\ge 0$ and $k\le p$. This follows by induction using
    \[\LL_{X/F}^k[-k]\to L\Omega_{X/F}^{\le k}\to L\Omega_{X/F}^{\le k-1}\xto{+1}\]
    \[\DB_{X/F}^k[-k]\to \DB_{X/F}^{\le k}\to \DB_{X/F}^{\le k-1}\xto{+1}\]
    and Proposition \ref{prop:surjectivity-over-C}. This finishes the proof of Claim 1.

    \vskip \medskipamount
    \noindent\textit{Claim 2:} The maps $\HH^{i+p}(\tilde{C}_j\LL_{X/F}^{p})\xto{\sim} \HH^{i+p}_{cdh}(\tilde{C}_j\DB_{X/F}^{p})$ are isomorphisms for all $i\ge -m+1$ and $p\ge 0$. 

    Indeed, this follows by induction using
    \[\tilde{C}_j\LL_{X/F}^p[-p]\to \tilde{C}_jL\Omega_{X/F}^{\le p}\to \tilde{C}_jL\Omega_{X/F}^{\le p-1}\xto{+1}\]
    \[\tilde{C}_j\DB_{X/F}^p[-p]\to \tilde{C}_j\DB_{X/F}^{\le p}\to \tilde{C}_j\DB_{X/F}^{\le p-1}\xto{+1}\]
    and Proposition \ref{prop:surjectivity-over-C}. 

    \vskip \medskipamount
    \noindent\textit{Claim 3:} The maps $\HH^{i+p}(\LL_{X/F}^{p})\xto{\sim} \HH^{i+p}_{cdh}(\DB_{X/F}^{p})$ are isomorphisms for all $i\ge -m+1$ and $p\ge 0$. 

    We have a split exact triangle (the splitting is induced by the zero section $s: X\to X\times \AA^j$):
    \[p^*\Omega^1_{X/F} \to \Omega^1_{X\times \AA^j/F}\to p^*\O_X\otimes \Omega_{X\times \AA^j/X}^1\xto{+1}\]
    Taking wedge product gives
    \[\Omega^p_{X\times \AA^j/F}\simeq \bigoplus_{i=0}^p p^*\Omega_X^i\otimes \Omega^{p-i}_{X\times \AA^j/X}\]
    (with the convention $\Omega_{X\times \AA^j/X}^p = \O_{X\times \AA^j}$ when $p=0$).

    Taking derived functor and cdh-sheafification gives
    \[\LL^p_{X\times \AA^j/F}\simeq \bigoplus_{i=0}^p p^*\LL_X^i\otimes \Omega^{p-i}_{X\times \AA^j/X}\]
    \[\DB^p_{X\times \AA^j/F}\simeq \bigoplus_{i=0}^p p^*\DB_X^i\otimes \Omega^{p-i}_{X\times \AA^j/X}\]
    Thus
    \[\HH^{i+p}(\tilde{C}_j\LL^p_{X\times \AA^j/F})\simeq \HH^{i+p}(\LL_X^p\otimes V_j)\oplus \bigoplus_{i=0}^{p-1} \HH^{i+p}(p^*\LL_X^i\otimes \Omega^{p-i}_{X\times \AA^j/X})\]
    \[\HH^{i+p}(\tilde{C}_j\DB^p_{X\times \AA^j/F})\simeq \HH^{i+p}(\DB_X^p\otimes V_j)\oplus \bigoplus_{i=0}^{p-1} \HH^{i+p}(p^*\DB_X^i\otimes \Omega^{p-i}_{X\times \AA^j/X})\]
    where $V_j = F[t_1,\cdots, t_j]/F$.

    Since the maps in claim 2 are compatible with this direct sum decomposition, the maps $\HH^{i+p}(\LL_{X/F}^{p})\xto{\sim} \HH^{i+p}_{cdh}(\DB_{X/F}^{p})$ are isomorphisms for all $i\ge -m+1$ and $p\ge 0$. 

    This completes the proof in one direction. For the converse, assume that the maps
    \[\HH^i(X,\LL_{X/F}^p)\xto{\sim} \HH^i(X, \DB_{X/F}^p)\]
    are isomorphisms for $i-p\ge -m+1$. We will show $X$ is $K_m$-regular. By (\ref{eqn:K-regularity-LES}) and the decomposition of cyclic homology, we need to show:
    \begin{enumerate}
        \item $\HH^{-m+1+2p}(X, \tilde{C}_jL\Omega_{X/\QQ}^{\le p})\xto{\sim} \HH^{-m+1+2p}_{cdh}(X,\tilde{C}_j\DB_{X/\QQ}^{\le p})$ are isomorphisms for all $p\ge 0$.
        \item $\HH^{-m+2p}(X, \tilde{C}_jL\Omega_{X/\QQ}^{\le p})\xto{\sim} \HH^{-m+2p}_{cdh}(X,\tilde{C}_j\DB_{X/\QQ}^{\le p})$ are surjective for all $p\ge 0$.
    \end{enumerate}

    Analyzing the spectral sequences
    \[E_1^{ij}=\Omega^i_{F/k}\otimes \tilde{C}_j\HH^j(X, L\Omega_{X/F}^{\le p-i})\implies \tilde{C}_j\HH^{i+j}(X, L\Omega_{X/\QQ}^{\le p})\]
    \[E_1^{ij}=\Omega^i_{F/k}\otimes \tilde{C}_j\HH^j(X, \DB_{X/F}^{\le p-i})\implies \tilde{C}_j\HH^{i+j}(X, \DB_{X/\QQ}^{\le p})\]
    shows that it is enough to show
    \begin{enumerate}
        \item $\HH^{-m+1+2p}(X, \tilde{C}_jL\Omega_{X/F}^{\le p})\xto{\sim} \HH^{-m+1+2p}_{cdh}(X,\tilde{C}_j\DB_{X/F}^{\le p})$ are isomorphisms for all $p\ge 0$.
        \item $\HH^{-m+2p}(X, \tilde{C}_jL\Omega_{X/F}^{\le p})\xto{\sim} \HH^{-m+2p}_{cdh}(X,\tilde{C}_j\DB_{X/F}^{\le p})$ are surjective for all $p\ge 0$.
    \end{enumerate}
    By definition of $\tilde{C}_j$, it suffices to show
    \begin{enumerate}
        \item $\HH^{-m+1+2p}(X, L\Omega_{X/F}^{\le p})\xto{\sim} \HH^{-m+1+2p}(X,\DB_{X/F}^{\le p})$ are isomorphisms for all $p\ge 0$.
        \item $\HH^{-m+1+2p}(X\times \AA^j, L\Omega_{X\times \AA^j/F}^{\le p})\xto{\sim} \HH^{-m+1+2p}(X\times \AA^j,\DB_{X\times \AA^j/F}^{\le p})$ are isomorphisms for all $p,j\ge 0$.
        \item $\HH^{-m+2p}(X\times \AA^j, L\Omega_{X\times \AA^j/F}^{\le p})\xto{\sim} \HH^{-m+2p}(X\times \AA^j,\DB_{X\times \AA^j/F}^{\le p})$ are surjective for all $p,j\ge 0$.
    \end{enumerate}
    Part (1) follows by assumptions. For (2) and (3), analyzing the spectral sequences
    \[E_1^{ij}= \HH^j(p^*L\Omega^i_{X/F}\otimes \Omega_{X\times \AA^j/X}^{\le p-i})\implies \HH^{i+j}(L\Omega_{X\times \AA^j/F}^{\le p})\]
    \[E_1^{ij}= \HH^j(p^*\DB^i_{X/F}\otimes\Omega_{X\times \AA^j/X}^{\le p-i})\implies \HH^{i+j}(\DB_{X\times \AA^j/F}^{\le p})\]
    we see that it suffices to show
    \begin{enumerate}
        \item $\HH^{-m+1+2p-k}(p^*L\Omega_{X/F}^k\otimes \Omega_{X\times \AA^j/X}^{\le p-k})\to \HH^{-m+1+2p-k}(p^*\DB_{X/F}^k\otimes \Omega_{X\times \AA^j/X}^{\le p-k})$ are isomorphisms for all $p\ge 0$ and $k\le p$.
        \item $\HH^{-m+2p-k}(p^*L\Omega_{X/F}^k\otimes \Omega_{X\times \AA^j/X}^{\le p-k})\to \HH^{-m+2p-k}(p^*\DB_{X/F}^k\otimes \Omega_{X\times \AA^j/X}^{\le p-k})$ are surjective for all $p\ge 0$ and $k\le p$.
    \end{enumerate}
    By induction using long exact sequences associated to
    \[\Omega_{X\times \AA^j/X}^{p-k}[-p+k]\to \Omega_{X\times \AA^j/X}^{\le p-k}\to \Omega_{X\times \AA^j/X}^{\le p-k-1}\xto{+1},\]
    it suffices to show 
    \begin{enumerate}
        \item $\HH^{-m+1+p}(L\Omega_{X/F}^k)\to \HH^{-m+1+p}(\DB_{X/F}^k)$ are isomorphisms for all $p\ge 0$ and $k\le p$.
        \item $\HH^{-m+p}(L\Omega_{X/F}^k)\to \HH^{-m+p}(\DB_{X/F}^k)$ are surjective for all $p\ge 0$ and $k\le p$.
    \end{enumerate}
    We have (1) by assumption, and (2) follows by Proposition \ref{prop:surjectivity-over-C}.    
\end{proof}

As an application, we construct an example that shows local $K_m$-regularity does not imply $K_m$-regularity in general.

\begin{example}\label{ex:locally-K-regular-does-not-imply-K-regular}
    For this example, we assume all varieties are defined over $\CC$, and we write $\DB_X^p$ for $\DB^p_{X/\CC}$.

    Let $X$ be any projective surface with $\H^1\DB_{X}^0\simeq \CC$ (for example, the projective cone over a curve of genus $2$), and $Y$ be a curve of genus $g\ge 1$, so that $H^1(Y,\O_Y)\cong \CC$. Take $Z=X\times Y$. Then there is an exact triangle
    \[\O_Z\to \DB_Z^0\to \O_{Y}[-1]\xto{+1}.\]
    Since $\H^2\DB_Z^0=0$, we see that $Z$ is locally $K_{-2}$-regular. However $Z$ is not $K_{-2}$-regular by Theorem \ref{thm:CHWW-projective}, because 
    \[\HH^2(\O_Y[-1])\simeq H^1(\O_Y)\neq 0.\]
\end{example}

For this reason, we do not expect pre-$m$-Du Bois varieties to be $K_{-d+m+2}$-regular in general. We summarize the relationships between these concepts when $m=0$ as follows:
% https://q.uiver.app/#q=WzAsNSxbMCwwLCJcXHRleHR7RHUgQm9pc30iXSxbMSwwLCJcXHRleHR7JEtfey1kKzJ9JC1yZWd1bGFyfSJdLFsyLDAsIlxcdGV4dHskS197LWQrMX0kLXJlZ3VsYXJ9Il0sWzEsMSwiXFx0ZXh0e2xvY2FsbHkgJEtfey1kKzJ9JC1yZWd1bGFyfSJdLFsyLDEsIlxcdGV4dHtsb2NhbGx5ICRLX3stZCsxfSQtcmVndWxhcn0iXSxbMCwxLCIvIiwzLHsibGV2ZWwiOjJ9XSxbMSwyLCIiLDMseyJsZXZlbCI6Mn1dLFszLDQsIiIsMyx7ImxldmVsIjoyfV0sWzMsMSwiLyIsMyx7ImxldmVsIjoyfV0sWzQsMiwiLyIsMyx7ImxldmVsIjoyfV0sWzAsMiwiIiwzLHsiY3VydmUiOi0zLCJsZXZlbCI6Mn1dLFswLDMsIiIsMyx7ImxldmVsIjoyfV1d
\[\begin{tikzcd}
	{\text{Du Bois}} & {\text{$K_{-d+2}$-regular}} & {\text{$K_{-d+1}$-regular}} \\
	& {\text{locally $K_{-d+2}$-regular}} & {\text{locally $K_{-d+1}$-regular}}
	\arrow["{/}"{marking, allow upside down}, Rightarrow, from=1-1, to=1-2]
	\arrow[Rightarrow, from=1-1, to=1-3, bend left = 20]
    \arrow[Rightarrow, from=1-1, to=2-2]
	\arrow[Rightarrow, from=1-2, to=1-3]
	\arrow["{/}"{marking, allow upside down}, Rightarrow, from=2-2, to=1-2]
	\arrow[Rightarrow, from=2-2, to=2-3]
	\arrow["{/}"{marking, allow upside down}, Rightarrow, from=2-3, to=1-3]
\end{tikzcd}\]

We conclude with the following example, which shows that Proposition \ref{prop:relation-general}(2) need not hold if $X$ is not affine, and has higher dimensional singular locus.
\begin{example}\label{ex:pre-DB-but-not-K-1-regular}
    We construct a 3-fold $X$ that is $K_{-1}$-regular, but not pre-$0$-Du Bois. It is non-affine with 1-dimensional singular locus. For this example, we assume all varieties are defined over $\CC$, and we write $\DB_X^p$ for $\DB^p_{X/\CC}$.

    Let $C$ be a genus 2 curve in $\PP^2$, and let $S=C\times \PP^1$. Denote by $\pi: S\to \PP^1$ the projection map. The line bundle
    \begin{align*}
        L &= \omega_S\otimes (\pi^*\omega_{\PP^1}(1))^{\vee} \\
        &= \omega_{S/\PP^1}\otimes (\pi^*\O_{\PP^1}(1))^{\vee}
    \end{align*}
    is relatively ample over $\PP^1$, so we can consider the relative affine cone 
    \[Y = {\bf Spec}_{\PP^1} \bigoplus_{m\ge 0} \pi_* L^m\]
    over $S$ defined by $L$. One checks that the arguments in \cite{K-cones} computing the Du Bois complexes of cones carry over to the relative case, giving that
    \[\H^1\DB_Y^0 \simeq \bigoplus_{m\ge 1} R^1 \pi_* L^m.\]
    is a sheaf supported on $Y_{\sing}\simeq \PP^1$. We have $R^1\pi_* L \simeq \O_{\PP^1}(-1)$ by Grothendieck duality, and $R^1\pi_* L^m =0$ by Kodaira vanishing. It follows that
    \[\H^1\DB_Y^0 \simeq \O_{\PP^1}(-1).\]
    Similarly, one checks that
    \[\H^1\DB_Y^1 \simeq \O_{\PP^1}(-1).\]
    Now, let $X$ be a projectivization of $Y$ that is smooth away from the singular points of $Y$. Then, we have
    \[\H^1\DB_X^0 \simeq \H^1\DB_X^1 \simeq \O_{\PP^1}(-1).\]
    In particular, $X$ is not pre-$0$-Du Bois. We claim that $X$ is $K_{-1}$-regular. Indeed, by Theorem \ref{thm:CHWW-projective}, this is equivalent to the following maps being isomorphisms: 
    \[\HH^3(\O)\xto{\sim} \HH^3(\DB_X^0)\]
    \[\HH^2(\O)\xto{\sim} \HH^2(\DB_X^0)\]
    \[\HH^3(\LL_X)\xto{\sim} \HH^3(\DB_X^1)\]
    Since the cone of the map $\O_X\to \DB_X^0$ is $\O_{\PP^1}(-1)[-1]$, which has no cohomology, we have the two isomorphisms. The cone of the map $\LL_X\to \DB_X^1$ is $\O_{\PP^1}(-1)[-1]\oplus C$, where the complex $C$ is supported on the 1-dimensional subvariety $X_{\sing}$, and in cohomological degrees $\le 0$. Analyzing the spectral sequence
    \[E_2^{ij} = H^i(\H^j C)\implies \HH^{i+j}(C)\]
    gives the last isomorphism.
\end{example}

\section{Further questions}
\noindent (1) As noted in \cite{K-DB}, being $K_m$-regular also implies \textit{strict $(m-1)$-Du Bois}, i.e. the maps
    \[\Omega_{X/\CC}^p\to \DB_{X/\CC}^p\]
    are isomorphisms for $p\le m-1$. Since there is no known example of a variety with strict-$1$-Du Bois singularities that is not a local complete intersection, this leads to the following question:
    \begin{question}
        Is there a $K_2$-regular variety that is not a local complete intersection?
    \end{question}
   \medskip 
\noindent (2) We showed in Theorem \ref{intro-thm:regularity} that $K_2$-regular affine local complete intersections are smooth. We do not know if this is the case when the variety is not necessarily affine.
    \begin{problem}\label{ques:K2-reg-non-affine}
        Give examples of singular non-affine varieties that are $K_2$-regular.
    \end{problem}

    One approach could be to construct projective varieties that are $K_2$-regular using Theorem \ref{thm:CHWW-projective}.

    \medskip

\noindent (3) The notion of \textit{higher-$F$-injectivity} is introduced in \cite{higher-F-inj} as an analogue of higher Du Bois singularities in positive characteristic, extending the well-known notion of $F$-injectivity (analogous to the Du Bois singularities for complex algebraic varieties). One can ask 
    \begin{question}
        Is there any relationship between higher $F$-injectivity and $K$-regularity in positive characteristic?
    \end{question}

    \iffalse
    \item Can we improve the bound $\codim supp \H^i\DB_X^j\le \dim X - j -1$ (when $i>0$) when $X$ has $m$-Du Bois singularities? One possible approach is to extend \cite{PSV}*{Proposition C} to local rings of a finte type $k$-algebra.
    \fi
\medskip
\noindent (4) In Section \ref{sec:DB-table}, we classified all possible Du Bois tables in dimension $2$. We ask if one can extend this classification to higher dimensions. 
\begin{question}
    What are all possible constraints in the Du Bois table in higher dimensions?
\end{question}

For threefolds over $\QQ$, the Du Bois table takes the form
\[
\begin{tabular}{|c c c c}
    $b^{0, 2}$ & $b^{1,2}$ & 0 &0 \\
    $b^{0, 1}$ & $b^{1,1}$ & $b^{2,1}$ &0 \\
    \hline
\end{tabular}
\]
Using the spectral sequence argument as in Lemma \ref{lemma:DB-table-constraint}, one can show that 
    \[b^{0,2}\ge b^{1,2} \quad\quad  b^{0,2}-b^{1,2}\ge b^{2,1}-b^{1,1}.\]
The question amounts to asking whether there are further relations among the entries in the Du Bois table. By the main theorem of \cite{K-Bass}, answers to this question will in particular give answers to Bass' question.

\end{document}